\documentclass[11pt]{article}

\usepackage{geometry}
\usepackage{color}
\usepackage{float}
\usepackage{textcomp}
\usepackage{amsthm}
\usepackage{amsmath}
\usepackage{amssymb}

\usepackage{latexsym}

\usepackage{changes}

\usepackage{mathtools}

\newtheorem{theorem}{Theorem}[section]
\newtheorem{lemma}[theorem]{Lemma}

\newtheorem{proposition}[theorem]{Proposition}

\newtheorem{assumption}[theorem]{Assumption}

\DeclareMathOperator{\argmin}{argmin}

\newcommand{\R}{\mathbb{R}}

\newcommand{\inner}[2]{\langle{#1},{#2}\rangle}

\newcommand{\tos}{\rightrightarrows} 
\newcommand{\Y}{\mathcal{Y}}
\newcommand{\X}{\mathcal{X}}
\newcommand{\Z}{\mathcal{Z}}

\newcommand{\M}{\mathcal{M}}

\newcommand{\vgap}{\vspace{.1in}}

\newcommand{\tz}{\tilde z}

\newcommand{\bi}{\begin{itemize}}
\newcommand{\ei}{\end{itemize}}
\newcommand{\ba}{\begin{array}}
\newcommand{\ea}{\end{array}}

\begin{document}

\title{Pointwise and ergodic convergence rates of a variable metric proximal ADMM}

\author{
    M.L.N. Gon\c calves
    \thanks{Instituto de Matem\'atica e Estat\'istica, Universidade Federal de Goi\'as, Campus II- Caixa
    Postal 131, CEP 74001-970, Goi\^ania-GO, Brazil. (E-mails: {\tt
       maxlng@ufg.br} and {\tt jefferson@ufg.br}).  The work of these authors was
    supported in part by  CNPq Grants 406250/2013-8, 444134/2014-0 and 309370/2014-0.}
    \and
     M. Marques Alves
\thanks{
Departamento de Matem\'atica,
Universidade Federal de Santa Catarina,
Florian\'opolis, Brazil, 88040-900 
({\tt maicon.alves@ufsc.br}).
The work of this author was partially 
supported by CNPq grants no.
406250/2013-8, 306317/2014-1 and 405214/2016-2.
}
    \and
 J.G. Melo \footnotemark[1]
}
\date{May 4, 2017}

\maketitle

\begin{abstract}
In this paper, we obtain global  $\mathcal{O} (1/ \sqrt{k})$  pointwise  and $\mathcal{O} (1/ {k})$ ergodic  convergence rates for  a variable metric proximal  alternating direction method of multipliers
(VM-PADMM) for solving   linearly constrained convex optimization problems. The VM-PADMM can be seen as a class of ADMM variants, allowing the use of degenerate metrics (defined by noninvertible linear operators).
We first propose and study nonasymptotic convergence rates of  a  variable metric hybrid proximal extragradient  
(VM-HPE) framework for solving monotone inclusions. Then, the above-mentioned convergence rates for the VM-PADMM are obtained essentially by showing that it 
falls within the latter framework. 
 To the best of our knowledge, this is the first time that global pointwise (resp. pointwise and ergodic) convergence rates are obtained for
the VM-PADMM  (resp. VM-HPE framework).
 \end{abstract}

%
%
%
%
%
%
%
%
%
%
%
%
%
%
%
%
%
%
%
%
%
%
%
%
%

\pagestyle{plain}

\section{Introduction} \label{sec:int}
We consider the  linearly constrained convex optimization problem
\begin{align}
 \label{optl1}
 \begin{aligned}
  &\text{minimize }    \;\;\;f(x) + g(y)&\\
  &\text{subject to }  \;\;Ax + By = b,\;&
 \end{aligned}
\end{align}
where $f:\X\to \overline{\mathbb{R}}:=\mathbb{R}\cup\{+\infty\}$ and $g:\Y\to \overline{\mathbb{R}}$ are extended-real-valued 
proper closed and convex functions,  $\X,\Y$ and $\Gamma$ are finite-dimensional real vector spaces, and $A:\X\to \Gamma$ and $B:\Y\to \Gamma$ are linear operators. 
One of the most popular methods for solving \eqref{optl1} is the alternating direction method of multipliers (ADMM)  \cite{Boyd:2011,0352.65034,0368.65053}, for which many variants have been proposed and studied in the literature; see, e.g., 
\cite{attouch:hal,doi:10.1137/13090849X,ChPo_ref2,Damek,Deng1,PADMM_Eckstein,GADMM2015,FPST_editor,MJR,Hager,he.lia.der-new.mp02,HeLinear,Lin,LanADMM}. 

In this paper, we obtain  global ergodic and pointwise convergence rates  for
a variable metric proximal ADMM (VM-PADMM) which  can be described as follows: given an initial point $(x_0,y_0,\gamma_0)\in \X\times\Y\times\Gamma$ and a stepsize $\theta>0$, compute  a sequence $\{(x_k,y_k,\gamma_k)\}$, recursively, by
\begin{align}
 \label{def:xk-padmm}
   x_k&\in \argmin_{x \in \X} \left \{ f(x) - \inner{\gamma_{k-1}}{Ax}_\X +
   \frac{1}{2} \| Ax +By_{k-1} -b \|^2_{\Gamma,H_k}+\frac{1}{2}\|x- x_{k-1}\|_{\X,R_k}^2 \right\},\\
\label{def:yk-padmm}
 y_k&\in \argmin_{y \in \Y} \left \{ g(y) - \inner{\gamma_{k-1}}{By}_\Y +
 \frac{1}{2} \| Ax_k+By -b\|^2_{\Gamma,H_k} +\frac{1}{2}\|y- y_{k-1}\|_{\Y,S_k}^2\right\},\\
  \label{def:gammak}
  \gamma_k& = \gamma_{k-1}-\theta H_k\left(Ax_k+By_k-b\right),
\end{align}
 where $H_k$, $R_k$ and $S_k$  are selfadjoint  linear operators such that $H_k$ is positive definite and $R_k$ and $S_k$ are  positive semidefinite, and  $\|\cdot\|_{\Gamma,H_k}^2:=\inner{ H_k (\cdot)}{\cdot}_\Gamma$, etc. We start by reviewing some existing methods and works related to the above method.
 \\[2mm]
 {\bf VM-PADMM and some  variants.}  The VM-PADMM \eqref{def:xk-padmm}--\eqref{def:gammak} can be seen as a class of ADMM variants, depending on the choices of the linear operators $H_k$, $R_k$ and $S_k$. Namely,
 \begin{itemize}
 \item by taking $H_k=\beta I$ with $\beta>0$, $R_k=0$, $S_k=0$ and $\theta=1$, it reduces to the standard ADMM,  whose the ergodic convergence rate was established in \cite{monteiro2010iteration};
 \item the  ADMM in \cite{HeLinear} (related to the Uzawa method \cite{Xahang}) consists of taking $H_k=\beta I$ with $\beta>0$, $R_k$ constant, $S_k=0$ and $\theta=1$. Pointwise and ergodic convergence rates for this variant   were obtained in \cite{HeLinear,He2015};

 \item  the proximal ADMM consists of choosing $H_k=\beta I$ with $\beta>0$, $R_k$ and $S_k$ constant. This method has been studied by many authors; see, for instance   \cite{Cui,Deng1,MJR2,HpeADMM}, where convergence rates are analyzed; 
 
  \item by choosing $H_k=\beta_k I$, $R_k=0$ and $S_k=0$, it corresponds to a variable penalty parameter ADMM, for which asymptotic convergence analysis was considered in  \cite{He1998,He2000,solodov1};
  
  \item the VM-PADMM \eqref{def:xk-padmm}--\eqref{def:gammak}  with $R_k$ and $S_k$ positive definite  is closely related to the method  studied in \cite{he.lia.der-new.mp02,solodov3} for solving (point-to-point) continuous monotone variational inequality problems (in the setting of problem \eqref{optl1},  it demands $f$ and $g$ to be continuously differentiable). We mention that, contrary to our analysis, the latter references consider the stepsize $\theta=1$ and do not present nonasymptotic convergence rates;

   \item by letting $H_k=\beta I$,  $\beta>0$, and $\theta=1$,  the resulting method becomes similar to Algorithm 7 in  \cite{Bot2016}, where a composite structure of $f$ is considered and ergodic convergence rates were obtained under  the additional conditions that $B=I$ in (\ref{optl1}) and the dual solution set of \eqref{optl1} be bounded.
    \end{itemize} 
 \noindent
{\bf Contributions of the paper.}
We obtain an $\mathcal{O}(1/k)$ global convergence rate for an ergodic sequence associated to the VM-PADMM~\eqref{def:xk-padmm}--\eqref{def:gammak} with $\theta\in (0,(\sqrt{5}+1)/2)$, which provides, for given tolerances $\rho,\varepsilon>0$, triples $(x,y,\tilde \gamma)$,  $(r_x,r_y,r_\gamma)$ and scalars $\varepsilon_x,\varepsilon_y\geq 0$ such that
\begin{align}
 \label{eq:bound_intro}
 \begin{aligned}
 r_x\in & \;\partial_{\varepsilon_x} f(x)-A^*\tilde\gamma, \quad r_y\in \partial_{\varepsilon_y} g(y)-B^*\tilde \gamma, 
\quad r_\gamma=Ax+By-b,\\
 \max&\left\{\|r_x\|^*_x\,,\;\|r_y\|^*_y\,,\;\|r_\gamma\|^*_\gamma\right\}\leq \rho,\\
\varepsilon_x+&\varepsilon_y\leq \varepsilon,
\end{aligned}
\end{align}
in at most $\mathcal{O}\left(\max\left\{\left\lceil d_0/\rho\,\right\rceil,
\left\lceil d_0^2/\varepsilon\,\right\rceil\right\}\right)$ iterations, where $\|\cdot\|^*_x\,,\;\|\cdot\|^*_y\,$ and $\;\|\cdot\|^*_\gamma$ denote dual seminorms associated to the linear operators $H_k,R_k$ and $S_k$, and $d_0$ is a scalar measuring the quality of the initial point. Moreover, we establish an $\mathcal{O}(1/\sqrt{k})$ pointwise convergence rate in which the inclusions in \eqref{eq:bound_intro} are strengthened, in the sense that $\varepsilon_x=\varepsilon_y=0$,  and the bound on the number of iterations becomes  $\mathcal{O}\left(\left\lceil d_0^2/\rho^2\,\right\rceil\right)$.  
Our study is done by first establishing global ergodic and pointwise convergence rates for a variable metric hybrid proximal extragradient (VM-HPE) framework for finding zeroes of  maximal monotone operators, and 
then by showing that the VM-PADMM \eqref{def:xk-padmm}--\eqref{def:gammak} can be seen as an instance of the latter framework. 
To the best of our knowledge, this is the first time that global pointwise (resp. pointwise and ergodic) convergence rates are obtained for
the VM-PADMM \eqref{def:xk-padmm}--\eqref{def:gammak} (resp. VM-HPE framework).
 Besides, our analysis allows degenerate metrics (induced by positive semidefinite linear operators) which makes the VM-PADMM \eqref{def:xk-padmm}--\eqref{def:gammak} (and the VM-HPE framework) more suitable for applications.  We next briefly review   some related works to the VM-HPE framework.
 \\[2mm]
 {\bf VM-HPE type frameworks.} 
 The  VM-HPE framework proposed in this work is a generalization of a special instance of the HPE framework \cite{Sol-Sv:hy.ext} allowing variations in the metric (induced  by positive semidefinite linear operators) along the iterations.   
The iteration complexity of the HPE framework was first analyzed in \cite{monteiro2010complexity} and subsequently applied to the study of several methods; see, for example,  \cite{YHe2,Maicon,monteiro2011complexity,monteiro2010iteration}.
 An inexact variable metric proximal point type method  was proposed in \cite{par.lot.sol-cla.sjo08} but, contrary to our VM-HPE framework, it demands the metrics to be nondegenerate (induced by {invertible} linear operators).  Moreover, the convergence analysis presented in \cite{par.lot.sol-cla.sjo08} does not 
include nonasymptotic convergence rates. 
\\[2mm]
{\bf Outline of the paper.} Subsection~\ref{sec:bas} presents our notation and basic results.   Section~\ref{sec:smhpe} introduces the
    VM-HPE framework and presents  its nonasymptotic pointwise and ergodic convergence rates, whose proofs are postponed to Appendix \ref{sec:app}.
Section~\ref{sec:amal} contains two subsections. In Subsection~\ref{sec:vmadmm.g},  we formally state the VM-ADMM \eqref{def:xk-padmm}--\eqref{def:gammak} and presents  its nonasymptotic  pointwise and ergodic convergence rates. 
In Subsection~ \ref{subsec:proof}, we  obtain the convergence rates of the VM-ADMM by viewing it as an instance of the VM-HPE framework.

\subsection{Basic results and notation}
\label{sec:bas}

Let $\Z$ be a finite-dimensional real vector space with inner
product $\inner{\cdot}{\cdot}_{\Z}$ and induced
norm $\|\cdot\|_{\Z}:=\sqrt{\inner{\cdot}{\cdot}_{\Z}}$.
Denote by $\M^{\Z}_{+}$ (resp. $\M^{\Z}_{++}$) the space of
selfadjoint positive semidefinite (resp. definite) linear operators
on $\Z$. 
Each element $M\in \M^{\Z}_{+}$ induces a symmetric bilinear
form $\inner{M(\cdot)}{\cdot}_{\Z}$ on $\Z \times \Z$ and
a  seminorm $\|\cdot\|_{\Z,M}:=\sqrt{\inner{M(\cdot)}{\cdot}_{\Z}}$ on $\Z$.
Since $\inner{M(\cdot)}{\cdot}_{\Z}$ is symmetric and bilinear, 
the following hold, for all $z,z'\in \Z$,
\begin{align}
\label{eq:545}
 \inner{z}{Mz'}&\leq \dfrac{1}{2}\|z\|^2_{\Z,M}+\dfrac{1}{2}\|z'\|^2_{\Z,M},\\
\label{eq:500}
 \|z+z'\|_{\Z,M}^2&\leq 2\left(\|z\|_{\Z,M}^2+\|z'\|_{\Z,M}^2\right).
\end{align}
Moreover, each $M\in \M^{\Z}_{+}$ also induces
a (extended) dual seminorm on $\Z$ defined by
\[
 \|z\|_{\Z,M}^*:=\sup_{\|z'\|_{\Z,M}\leq 1}\,\inner{z}{z'}_{\Z}\qquad (z\in \Z).
\]
On the other hand, each $M\in \M^{\Z}_{++}$ induces an inner product
$\inner{M(\cdot)}{\cdot}_{\Z}$ and a norm 
$\|\cdot\|_{\Z,M}:=\sqrt{\inner{M(\cdot)}{\cdot}_{\Z}}$ on $\Z$, etc.

Next two propositions, whose proofs are omitted, will be useful
in this paper.

\begin{proposition}
\label{pr:dn}
 For every $M\in \M_+^{\Z}$, we have $\emph{dom}\,\|\cdot\|^*_{\Z,M}=\mathcal{R}(M)$ and $\|M(\cdot)\|^*_{\Z,M}=\|\cdot\|_{\Z,M}$, where $\mathcal{R}(M)$ denotes  the range of $M$.
\end{proposition}
\noindent
Let the partial order $\preceq$ on $\M^{\Z}_{+}$ be defined by
\[
  M\preceq N \;\iff\; N-M\in \M^{\Z}_{+}.
\]
\begin{proposition}
 \label{pr:dns}
Let $M,N \in \M^{\Z}_{+}$ and $c>0$. If $M\preceq cN$, then
\begin{align}
 \label{eq:dnorm}
 \|\cdot\|_{\Z,M}\leq\sqrt{c}\,\|\cdot\|_{\Z,N}\;\;\emph{and}\;\;\|\cdot\|_{\Z,N}^* \leq\sqrt{c}\,\|\cdot\|_{\Z,M}^*.
\end{align}
\end{proposition}

A set-valued mapping $T:\Z\tos \Z$
 %
%
 is said to be \emph{monotone} if 
\[
 \inner{v-v'}{z-z'}\geq 0\quad \forall \; z,z'\in \Z, \forall \;v\in T(z), \forall \; v'\in T(z').
\]
Moreover, $T$ is \emph{maximal monotone} if it is monotone and, additionally, if $S$ is a monotone operator such that $T(z)\subset S(z)$ for every $z\in \Z$  then $T=S$.
 The \emph{inverse} operator $T^{-1}:\Z\tos \Z$ of $T$ is given by
 $T^{-1}(v):=\{z\in \Z\;|\; v\in T(z)\}$.
Given  $\varepsilon\geq0$, the 
 {$\varepsilon$-enlargement} $T^{\varepsilon}:\Z\tos \Z$
 of a set-valued mapping $T:\Z\tos \Z$ is defined as
\[
 T^{\varepsilon}(z)
 :=\{v\in \Z\;|\;\inner{v-v'}{z-z'}\geq -\varepsilon,\;\;\forall z' \in \Z, \forall\; v'\in T(z')\} \quad \forall z \in \Z.
\]

Recall that the 
{$\varepsilon$-subdifferential} of a 
 convex function $f:\Z\to  \overline{\mathbb{R}}$
is defined by
$\partial_{\varepsilon}f(z):=\{v\in \Z\,|\,f(z')\geq f(z)+\inner{v}{z'-z}-\varepsilon\;\;\forall z'\in \Z\}$ 
for every $z\in \Z$.
When $\varepsilon=0$, then $\partial_0 f(z)$ 
is denoted by $\partial f(z)$
and is called the \emph{subdifferential} of $f$ at $z$.
The operator $\partial f$ is trivially monotone if $f$ is proper.
If $f$ is a proper closed and convex function, then
$\partial f$ is also maximal monotone~\cite{Rockafellar}.

The following result is a particular case of the {\it weak transportation formula} in \cite[Theorem~2.3]{Bu-Sag-Sv:teps1} combined with \cite[Proposition~2(i)]{Bu-Iu-Sv:teps}.
\begin{theorem}
 \label{th:tf}
  Suppose $T:\Z\tos \Z$ is maximal monotone and
	let $\tilde z_i, r_i\in \Z$,
	for $i=1,\dots, k$, be such that
$
	 r_i\in T(\tilde z_i)
	$
	and define
	\[
	 \tilde z_k^a:=\frac1k\sum_{i=1}^k\, \tilde z_i\,,\quad r_k^a:=\frac1k\sum_{i=1}^k\; r_i\,,\quad
	 \varepsilon_k^a:=\frac1k\sum_{i=1}^k\inner{r_i}{\tilde z_i-\tilde z_k^a}.
	\]
	 %
	Then, the following hold:
	 \begin{itemize}
	 \item[\emph{(a)}] $\varepsilon_k^a\geq 0$ and $r_k^a\in T^{\varepsilon_k^a}(\tilde z_k^a)$;
	  \item[\emph{(b)}] if, in addition, $T=\partial f$ for some proper closed and convex function
		$f$,
		then $r_k^a\in \partial_{\varepsilon_k^a} f(\tilde z_k^a)$. 
	\end{itemize}
	\end{theorem}

\section{A variable metric HPE framework}
\label{sec:smhpe}
Consider the monotone inclusion problem 
%
\begin{align}
\label{eq:inc.p}
 0\in T(z),
\end{align}
where $\Z$ is a finite-dimensional inner product 
real vector space and $T:\Z\tos \Z$ is maximal monotone. 
Assume that the solution set $T^{-1}(0)$ of~\eqref{eq:inc.p}  
is nonempty. 
%

In this section, we propose a variable metric hybrid proximal extragradient
(VM-HPE) framework for solving 
\eqref{eq:inc.p} and 
analyze  its nonasymptotic  convergence rates. %
The proposed framework finds its roots
in the hybrid proximal extragradient (HPE) framework of \cite{Sol-Sv:hy.ext}, for which
the iteration complexity was recently obtained  in~\cite{monteiro2010complexity}.
%
Our main results
on pointwise and ergodic  convergence rates for the VM-HPE framework  are presented in Theorems
\ref{th:alpha} and \ref{th:vhpe.eg}, respectively.
In Section \ref{sec:amal}, we will show how the VM-HPE framework can be used to analyze the nonasymptotic 
 convergence  of a VM-PADMM  for solving linearly constrained convex optimization problems. 

\vgap
We  begin by stating the VM-HPE framework.
\vgap
\vgap

\noindent
\fbox{
\begin{minipage}[h]{6.4 in}
{\bf A  variable metric hybrid proximal extragradient (VM-HPE) framework}
\begin{itemize}
\item[(0)] Let $z_0 \in \Z$, $\eta_0 \in \R_{+}$ and  $\sigma \in [0, 1)$ 
be given, and set $k=1$.
\item[(1)] Choose  $M_k\in \M^{\Z}_{+} $ and find   $( z_k,\tilde{z}_k,\eta_k) \in \Z \times \Z \times  \mathbb{R}_{+}$ 
   such that 
     \begin{align}\label{breg-subpro}
&  r_k:=M_k(z_{k-1}-{z_k})  \in  T(\tz_k), \\[2mm]
& \|{z_k}- {\tz}_k\|_{\Z,M_k}^2  +\eta_k \leq \sigma \|{z_{k-1}} -\tz_{k}\|_{\Z,M_k}^2+\eta_{k-1}\label{breg-cond1}.
  \end{align}

\item[(2)] Set $k\leftarrow k+1$ and go to step 1.
\end{itemize}
\noindent
{\bf end}
\end{minipage}
}
\vgap
\vgap
\\
{\bf Remarks}. 1) 
Letting $M_k\equiv I$ and $\eta_k\equiv 0$ in \eqref{breg-subpro}
and \eqref{breg-cond1}, respectively, we find
that the sequences $\{z_k\}$, $\{\tilde z_k\}$ and $\{r_k\}$
satisfy
\begin{align*}
 &r_k\in T(\tilde z_k),\quad \|r_k+\tilde z_k-z_{k-1}\|_{\Z}^2\leq \sigma\|\tilde z_k-z_{k-1}\|_{\Z}^2,\\
 &z_k=z_{k-1}-r_k,
\end{align*}
which is to say that in this case the VM-HPE framework reduces to a special case of the HPE framework (see pp. 2763 in \cite{monteiro2010complexity})
with $\lambda_k\equiv 1$ (in the notation of \cite{monteiro2010complexity}) or, 
in other words, the VM-HPE framework is a generalization of a special case of the HPE framework in
which variations in the metric are allowed along the iterations. 
2)  If the sequence 
$\{M_k\}_{k\geq 0}$ is taken to be constant, then the VM-HPE framework reduces to a special case of the NE-HPE framework studied in \cite{MJR2}.
3) We also mention that a variable metric inexact proximal point method with relative error tolerance was
proposed in \cite{par.lot.sol-cla.sjo08} but, contrary to our framework, the method
of \cite{par.lot.sol-cla.sjo08} demands that every operator $M_k$ must be positive definite.
 Moreover, the convergence analysis presented in \cite{par.lot.sol-cla.sjo08} does not 
include nonasymptotic convergence rates. The fact that the  VM-HPE framework allows
positive semidefinite operators $M_k$ will be crucial for viewing the VM-PADMM 
of Section \ref{sec:amal} as a special instance of it.

From now on in this section, we assume the following condition to hold:
\begin{assumption}
 \label{ass:a1}
 For the sequence $\{M_k\}_{k\geq 1}$ generated by the \emph{VM-HPE framework}, there exist $M_0\in \M^{\Z}_{+}$, $0\leq C_S<\infty$ and, for each $k\geq 0$, $c_k\geq 0$ 
 such that $\{c_k\}_{k\geq 0}$ and  $\{M_k\}_{k\geq 0}$ 
 satisfy 
\begin{equation}
\label{CS}
  \sum_{i=0}^{k}\, c_i\leq C_S, \qquad \frac{1}{1+c_k}M_k\preceq M_{k+1}\preceq (1+c_{k})M_{k}
	\quad \forall \;k\geq 0.
\end{equation}
\end{assumption}
\noindent
{\bf Remark.} The above assumption  (which is similar to  condition (1.4) in \cite{par.lot.sol-cla.sjo08}) is satisfied, for instance, if the sequence 
$\{M_k\}_{k\geq 0}$ is taken to be constant and $c_k\equiv 0$, in which case  one can choose $C_S=0$. 

It is easy to check that Assumption \ref{ass:a1} implies the existence of a constant
$C_P>0$ such that $\{c_k\}_{k\geq 0}$ and  $\{M_k\}_{k\geq 0}$ satisfy
\begin{equation}
 \label{pro:Q12} 
   \prod_{i=0}^{k}(1+c_i)\leq C_P\;\;\;\mbox{and}\;\;\;
   M_{j}\preceq C_P M_k, \quad \forall \,j,k\geq 0.
 \end{equation}

In the remaining part of this section, we present pointwise and ergodic
 convergence rates for the VM-HPE framework. These results will depend
on the quantity:
 \begin{align}
  \label{def:d0}
   d_0 := \inf \{ \|z^*-z_{0}\|_{\Z,\,M_0} \;|\; z^* \in T^{-1}(0)\},
\end{align}
which measures the ``quality'' of the initial guess $z_0\in \mathcal{Z}$ in the VM-HPE framework
with respect to the solution set $T^{-1}(0)$.

For technical reasons and for the convenience of the reader,  the proofs of the next two theorems  
will be given in Appendix \ref{sec:app}.

\begin{theorem} {\bf (Pointwise convergence rate of the VM-HPE framework)} 
\label{th:alpha}
 Let $\{\tilde z_k\}$, $\{r_k\}$ and $\{M_k\}$ be generated by the \emph{VM-HPE framework}. Let also
 $C_P$ and $d_0$ be as in \eqref{pro:Q12} and  \eqref{def:d0}, respectively.
 Then, for every $k \ge 1$, 
 $r_k\in T(\tilde z_k)$ and there exists $i\leq k$ such that
 \begin{align}
    \label{v_ieps_i-bound-a}
	  &\|r_i\|^*_{\Z,M_i} \leq 
    \left(\frac{2(1+\sigma)C_P(d_0^2  +\eta_{0})+2(1-\sigma)\eta_{0}}{(1-\sigma)k}\right)^{1/2}. 
	\end{align}
\end{theorem}

\noindent
{\bf Remarks}. 1) If $c_k\equiv0$ in Assumption \ref{ass:a1} (in which case $M_k\equiv M_0$), then the upper bound in \eqref{v_ieps_i-bound-a} with $C_S=0$ and $C_P=1$ reduces essentially to a special case of  \cite[Theorem 3.3(a)]{MJR2} (with $\lambda_k\equiv 1, \varepsilon_k\equiv 0$ and $d(w)_z(z')=(1/2)\|z-z'\|^2$). Additionally, if $M_0=I$ and $\eta_0=0$, then the bound \eqref{v_ieps_i-bound-a}  becomes similar to the corresponding one in  \cite[Theorem 4.4(a)]{monteiro2010complexity}. 2) For a given tolerance $\rho>0$, Theorem \ref{th:alpha} ensures that there exists an index
\begin{align}
 \label{eq:104}
i=\mathcal{O}\left(\left\lceil\dfrac{C_p(d_0^2+\eta_0)}{\rho^2}\right\rceil\right)
\end{align}
such that
 \begin{align}
 \label{eq:103}
 r_i\in T(\tilde z_i)\;\;\mbox{and}\;\;\|r_i\|_{\Z,M_i}^*\leq \rho.
\end{align}
In this case, $\tilde z_i\in \Z$ can be interpreted  as a $\rho$-approximate solution 
of \eqref{eq:inc.p} with residual $r_i\in \Z$ (see, e.g., \cite{monteiro2010complexity} for the
definition of a related concept). 3) Although $M_i$ may not be invertible, criterion \eqref{eq:103} makes sense due to the fact that $r_i$ belongs to the image of $M_i$ (see \eqref{breg-subpro}). Indeed, if  $\|r_i\|_{\Z,M_i}^*=0$, then \eqref{breg-subpro} and Proposition~\ref{pr:dn} imply that $r_i=0$, and hence it follows from \eqref{eq:103}  that $\tilde z_i$ is a solution of problem~\eqref{eq:inc.p}.

Before presenting the ergodic convergence of the VM-HPE framework, let us define 
the ergodic sequences $\{\tilde z_k^a\}$, $\{r_k^a\}$
and $\{\varepsilon_k^a\}$ associated to
$\{\tilde z_k\}$ and $\{r_k\}$ as follows:
\begin{align}
 \label{SeqErg}
\tilde z^a_{k} := \frac{1}{{k}} \sum_{i=1}^k \tilde z_i,\quad  
 r^a_{k} := \frac{1}{{k}} \sum_{i=1}^k r_i, \quad
\varepsilon^a_{k} := \frac{1}{{k}} \sum_{i=1}^k   \inner{r_i}{\tilde z_i -\tilde z^a_{k}}.
\end{align}
%


\begin{theorem} {\bf (Ergodic convergence rate  of  the VM-HPE framework)}
\label{th:vhpe.eg}
Let $\{\tilde z_k^a\}$, $\{r_k^a\}$ and $\{\varepsilon_k^a\}$ be given as in
\eqref{SeqErg} and $\{M_k\}$ be generated by the \emph{VM-HPE framework}. Let also  $C_S$, $C_P$ and $d_0$ be as in  \eqref{CS}, \eqref{pro:Q12} and  \eqref{def:d0}, respectively.
Then, for every $k\geq 1$, we have $r^a_k \in T^{\varepsilon^a_{k}}(\tz^a_k)$ and
\begin{align}
  \label{th:vhpe.eg02} 
  &\|r_k^a\|^*_{\Z,M_k}  \leq \dfrac{\mathcal{E}\sqrt{d_0^2+\eta_0}}{k},\\
  \label{th:vhpe.eg03}
	&0\leq \varepsilon^a_{k} \leq\frac{\widehat{\mathcal{E}} (d_0^2+\eta_0)}{k},
\end{align}
where 
$\mathcal{E}:= (1+C_P)\left(\sqrt{C_P}+C_SC_P\right)+C_SC_P^{3/2}$ and \,$\widehat{\mathcal{E}} :=2C_P(1+C_S)\left[{\sigma C_P}/{(1-\sigma)}+2(1+C_P)\right]$.
\end{theorem}
\noindent
{\bf Remarks.} 
1) Similarly to the first remark after Theorem~\ref{th:alpha}, Theorem~\ref{th:vhpe.eg} is also related to  \cite[Theorem 3.4]{MJR2} and  \cite[Theorem 4.7]{monteiro2010complexity}.  
2) For given tolerances $\rho,\varepsilon>0$, Theorem~\ref{th:vhpe.eg}  ensures that in at most
\begin{align}
 \label{eq:1104}
 \mathcal{O}\left((1+C_S)C_p^2\max\left\{\left\lceil\dfrac{\sqrt{d_0^2+\eta_0}}{\rho}\,\right\rceil,
\left\lceil\dfrac{d_0^2+\eta_0}{\varepsilon}\,\right\rceil
\right\}\right)
\end{align}
iterations there hold
 \begin{align}
 \label{eq:1103}
 r^a_k\in T^{\varepsilon_k^a}(\tilde z^a_k),\;\;\|r^a_k\|_{\Z,M_k}^*\leq \rho \;\;\mbox{and} \;\;\varepsilon_k^a\leq \varepsilon.
\end{align}
 Note that  \eqref{eq:1104}, in terms of the  dependence on $\rho>0$, is better than the bound in \eqref{eq:104} by a factor of 
$\mathcal{O}\left(\rho\right)$ but, on the other
hand, since $\varepsilon_k^a$ can be strictly positive, the inclusion in \eqref{eq:1103} is potentially weaker than
the one in \eqref{eq:103}. 


\section{A  variable metric proximal  alternating direction method of multipliers}
\label{sec:amal}
This section contains two subsections. In Subsection \ref{sec:vmadmm.g},
we formally state the VM-PADMM
\eqref{def:xk-padmm}--\eqref{def:gammak} and  present its nonasymptotic  convergence rates. The main results are
Theorems \ref{th:maintheoADMM} and \ref{th:ergodicPADMM} in which   pointwise
and ergodic  convergence rates are obtained, respectively. The proofs of the latter theorems are discussed separately in Subsection \ref{subsec:proof} by
viewing the method as an  instance of the VM-HPE framework and by applying the results of Section~\ref{sec:smhpe}.

\subsection{VM-PADMM and its  convergence rates}
\label{sec:vmadmm.g}
Let $\X$, $\Y$ and $\Gamma$ be finite-dimensional real inner product vector spaces.
Consider the convex optimization problem \eqref{optl1}, i.e.,
\begin{align}
 \label{optl}
 \begin{aligned}
  &\text{minimize }    \;\;\;f(x) + g(y)&\\
  &\text{subject to }  \;\;Ax + By = b,\;&
 \end{aligned}
\end{align}
where the following assumptions are assumed to hold:
\begin{itemize}
\item[\bf(O1)] $f: \X \to \overline{\mathbb{R}}$ and $g:\Y \to \overline{\mathbb{R}}$ are proper
closed and convex functions;
\item[\bf(O2)] $A: \X \to \Gamma$ and $B: \Y \to  \Gamma$ are linear operators and $b \in \Gamma$;
\item[\bf(O3)] the solution set of \eqref{optl} is nonempty.
\end{itemize}
%
Under the above assumptions and standard constraint qualifications (see, e.g.,\cite[Corollaries~28.2.2 and 28.3.1]{Rockafellar70}), a vector $(x^*,y^*)\in \X\times \Y$ is a solution of
\eqref{optl} if and only if there exists a (Lagrange multiplier) $\gamma^*\in \Gamma$ 
such that $(x^*,y^*,\gamma^*)$ is a solution of
\begin{align}
 \label{eq:fooc}
0 \in   \partial f(x)- A^*{\gamma}, \quad 0 \in \partial g(y)- B^*{\gamma}, \quad Ax+By-b=0.
\end{align}
Motivated by the above statement, 
we define 
\begin{align}
 \Omega^*:=\left\{(x^*,y^*,\gamma^*)\in \X\times \Y\times \Gamma\;|\;
(x^*,y^*,\gamma^*)\;\mbox{is a solution of}\;\eqref{eq:fooc}\right\},\label{def:ess}
\end{align}
which is assumed to be nonempty.

The  convergence rates of the VM-PADMM (stated below)  for solving \eqref{optl} will be obtained by viewing  the optimization problem \eqref{optl}
as the  monotone inclusion \eqref{eq:fooc}, which is associated to a certain maximal monotone operator (see \eqref{FAB})
in $\X\times \Y\times \Gamma$, and by applying the results of the previous section.


\vgap
\vgap

\noindent
\fbox{
\begin{minipage}[h]{6.4 in}
{\bf Variable metric proximal  alternating direction method of multipliers (VM-PADMM).}
\begin{itemize}
\item[(0)] Let $(x_0,y_0,\gamma_0) \in \X\times \Y\times  {\Gamma}$ and $\theta\in (0,(\sqrt{5}+1)/2)$  be given, and set $k=1$.
\item[(1)]  Choose $R_k\in \M^{\X}_{+}$, $S_k\in \M^{\Y}_{+}$ and $H_k\in \M^{\Gamma}_{++}$
            and  compute an optimal solution $x_k \in \X$  of the subproblem
 \begin{equation} 
 \label{def:tsk-admm}
   \min_{x \in \X} \left \{ f(x) - \inner{ {\gamma}_{k-1}}{Ax}_\X +
   \frac{1}{2} \| Ax +By_{k-1} -b \|^2_{\Gamma,H_k}+\frac{1}{2}\|x- x_{k-1}\|_{\X,R_k}^2 \right\}
\end{equation}
 and compute an optimal solution $y_k\in \Y$ of the subproblem
\begin{equation} 
\label{def:tyk-admm}
 \min_{y \in \Y} \left \{ g(y) - \inner{ {\gamma}_{k-1}}{By}_\Y +
 \frac{1}{2} \| Ax_k+By -b\|^2_{\Gamma,H_k} +\frac{1}{2}\|y- y_{k-1}\|_{\Y,S_k}^2\right\}.
\end{equation}
\item[(2)] Set 
\begin{equation}
  \label{admm:eqxk}
  \gamma_k = \gamma_{k-1}-\theta H_k\left(Ax_k+By_k-b\right),
\end{equation}
$k \leftarrow k+1$, and go to step~(1).
\end{itemize}
\noindent
{\bf end}
\end{minipage}
}

\vgap
\noindent
{\bf Remarks}. 1)  As already mentioned in Section \ref{sec:int}, the VM-PADMM can be regarded as a class of ADMM instances, allowing a unified study of different variants of ADMM.
 2) An usual choice for the linear operator $H_k$ is $\beta_k I$, where $\beta_k>0$ plays the role
of a penalty parameter. 3) The proximal terms in \eqref{def:tsk-admm} and \eqref{def:tyk-admm} defined by $R_k$ and $S_k$, respectively,  may have different roles. Namely, they can be used to regularize the subproblems in
\eqref{def:tsk-admm} and \eqref{def:tyk-admm}, making them strongly convex (when $R_k$ and $S_k$ are positive definite operators) and hence admitting unique solutions. Moreover,  
by a careful choice of these operators,  subproblems \eqref{def:tsk-admm} and \eqref{def:tyk-admm} may become much easier to solve; for instance, if $H_k=\beta_k I$, then $R_k=\tau_k I-\beta_k A^*A$ with $\tau_k>\beta_k\|A^*A\|$ and  $S_k=s_kI-\beta_k B^*B$ with $s_k>\beta_k\|B^*B\|$ eliminate the presence of quadratic forms associated to $A^*A$ and $B^*B$ in \eqref{def:tsk-admm} and \eqref{def:tyk-admm}, respectively. 

%

From now on in this section, the following conditions are assumed
to hold:
%

\begin{assumption}
 \label{ass:b1} 
For the sequences $\{R_k\}_{k\geq 1}$, $\{S_k\}_{k\geq 1}$ and $\{H_k\}_{k\geq 1}$
generated by the \emph{VM-PADMM}, there exist $R_0\in \M_{+}^{\X}$, $S_0\in \M_{+}^{\Y}$, $H_0\in \M_{++}^{\Gamma}$, $0\leq C_S<\infty$
and, for each $k\geq 0$,  $c_k\in[0,1]$
such that $\{c_k\}_{k\geq 0}$,
 $\{Q_{k,1}:=R_k\}_{k\geq 0}$, $\{Q_{k,2}:=S_k\}_{k\geq 0}$ and $\{Q_{k,3}:=H_k\}_{k\geq 0}$ satisfy
\begin{align}
 \label{CS2}
  \sum_{i=0}^{k} c_i\leq C_S,\quad 
\frac{1}{1+c_k}Q_{k,j}\preceq Q_{k+1,j}\preceq (1+c_{k})Q_{k,j} \quad \forall k\geq 0,\;\; j=1,2,3.
\end{align}
\end{assumption}
Analogously to condition \eqref{pro:Q12}, assumption \ref{ass:b1} implies the
existence of $C_P>0$ such that $\{c_k\}_{k\geq 0}$ satisfies
\begin{align}
 \label{pro:Q122}
 \prod_{i=0}^k\,(1+c_i)\leq C_P\qquad \forall k\geq 0.
\end{align}
\noindent
We mention that Assumption \ref{ass:b1} is similar to Condition C in \cite{he.lia.der-new.mp02} but, contrary to 
the latter reference, none of the operators $R_k$ and $S_k$  is assumed to be positive definite.

Similarly to the previous section, the following quantity will be needed:
{\small
\begin{equation}
 \label{def:d0admm}
d_0:=\inf \left\{ \left(\|x_0-x^*\|^2_{\X,R_0}+\|y_0-y^*\|_{\Y,(B^*H_0B+S_0)}^2+\|\gamma_0-\gamma^*\|^2_{\Gamma,\theta^{-1}H_0^{-1}}\right)^{1/2} \;|\;(x^*,y^*,\gamma^*)\in \Omega^*\right\},
\end{equation}}
where $(x_0,y_0,\gamma_0)$ and $\theta$ are given in Step~(0) of the VM-PADMM, $R_0\in \M_{+}^{\X}$, $S_0\in \M_{+}^{\Y}$ and $H_0\in \M_{++}^{\Gamma}$
 are given in Assumption \ref{ass:b1}, and $\Omega^*$ is defined in \eqref{def:ess}.
 
\vgap
Next we present the two main results of this paper, whose proofs are 
given in  Subsection~\ref{subsec:proof}.

\begin{theorem}{\bf (Pointwise convergence rate of the VM-PADMM)}
 \label{th:maintheoADMM} 
Let $\{(x_k,y_k,\gamma_k)\}$, $\{R_k\}$, $\{S_k\}$ and $\{H_k\}$ be generated by the \emph{VM-PADMM}  
and let 
\begin{align}
 \label{xtilde}
   \tilde{\gamma}_{k}:={\gamma}_{k-1}-H_k(Ax_k+By_{k-1}-b)\quad \forall\,k\geq 1.
 \end{align}
Let also $C_P$
and $d_0$ be as in \eqref{pro:Q122} and \eqref{def:d0admm}, respectively.
Then, there exists a parameter $\sigma_\theta\in (0,1)$ such that, for all $k\geq 1$,
\begin{equation}
 \label{eq:main.pc.admm} 
 \left(
\begin{array}{c}
 r_{k,x}\\[1mm]
 r_{k,y}\\[1mm]
 r_{k,\gamma}\\[1mm]
\end{array} 
 \right)
:= 
\left( 
 \begin{array}{c} 
 R_k (x_{k-1}-x_k)\\[1mm]  
 (B^*H_kB+S_k)(y_{k-1}-y_k)\\[1mm]  
 \theta^{-1}H_k^{-1}(\gamma_{k-1}-\gamma_k)
 \end{array} 
 \right) \in 
 \left( 
 \begin{array}{c} 
 \partial f(x_k)-A^*\tilde{\gamma}_k\\\partial g(y_k)-B^*\tilde{\gamma}_k\\Ax_k+By_k-b
 \end{array} \right)
 \end{equation}
and, for some  $i\leq k$,
{\small
\begin{align}
 \label{eq:main.pc.admm02}
\max\left\{\|r_{i,x}\|^*_{\X,R_i}\,,\;\|r_{i,y}\|^*_{\Y,(B^*H_iB+S_i)}\,,\;\|r_{i,\gamma}\|^*_{\Gamma,\theta^{-1}H_i^{-1}}\right\}
	\leq 
	\frac{{d_0}}{\sqrt{k}} \sqrt{ \frac{2(1+\sigma_{\theta})C_P(1+\tau_{\theta})+2(1-\sigma_{\theta})\tau_{\theta}}{(1-\sigma_{\theta})}},
\end{align}}
  where $\tau_\theta:=(8(\sigma_\theta+\theta-1)\max\{1,\theta/{(2-\theta)}\})/\sqrt{\theta^3}$.
%
%
\end{theorem}

\vgap
\noindent
{\bf Remark.} For a given tolerance $\rho>0$,
Theorem \ref{th:maintheoADMM} guarantees the existence of triples 
$(x,y,\tilde \gamma)$, $(r_x,r_y,r_\gamma)$ and operators $R\in \M_+^{\X}$, $S\in \M_+^{\Y}$ and $H\in \M_{++}^{\Gamma}$ 
(generated by the VM-PADMM)
such that
\begin{align}
 \label{eq:comp.rho.e}
 \begin{aligned}
 r_x \in \;& \partial f(x)-A^*\tilde\gamma, \quad r_y\in \partial g(y)-B^*\tilde \gamma, \quad r_\gamma=Ax+By-b,\\[2mm]
 \max&\left\{\|r_x\|^*_{\X,R}\,,\;\|r_y\|^*_{\Y,(B^* HB+S)}\,,\;\|r_\gamma\|^*_{\Gamma,\theta^{-1} H^{-1}}\right\}\leq \rho,
\end{aligned}
\end{align}
in at most 
\begin{align}
 \label{eq:comp.rhoe2}
  \mathcal{O}\left(\left\lceil\dfrac{C_p\,d_0^2}{\rho^2}\right\rceil\right)
\end{align}
iterations, where $C_P$ and $d_0$ are as in \eqref{pro:Q122} and \eqref{def:d0admm}, respectively. The triple $(x,y,\tilde \gamma)$
in  \eqref{eq:comp.rho.e}  can be seen as a $\rho$-approximate solution of the KKT system \eqref{eq:fooc}
with residual $(r_x,r_y,r_\gamma)$.

Before proceeding to present the ergodic convergence of the VM-PADMM we need
to introduce its associated ergodic sequences.
Let $\{(x_k,y_k,\gamma_k)\}$ be generated by the VM-PADMM, let $\{\tilde \gamma_k\}$
and $\{(r_{k,x},r_{k,y},r_{k,\gamma})\}$  be defined as  in \eqref{xtilde} and \eqref{eq:main.pc.admm}, respectively,
and let the \emph{ergodic} sequences associated to them  be defined by
\begin{align}
 \label{def:erg}
 &(x_k^a,y_k^a):= \frac1k \sum_{i=1}^k\left(x_i, y_i\right),\quad \tilde \gamma_k^a:=\frac1k \sum_{i=1}^k\,\tilde \gamma_i,\\
\label{def:erg2} 
&(r^a_{k,x},r^a_{k,y},r^a_{k,\gamma}):=\frac1k\sum_{i=1}^k(r_{i,x},r_{i,y},r_{i,\gamma}),\\
\label{def:erg3b}
&(\varepsilon^a_{k,x},\varepsilon^a_{k,y})
:= \frac{1}{{k}}\sum_{i=1}^k\big(\,\inner{r_{i,x}+A^*\tilde \gamma_i}{x_i-x_k^a}_{\X},\; 
\inner{r_{i,y}+B^*\tilde \gamma_i}{y_i-y_k^a}_{\Y}\,\big).
\end{align}
 

\begin{theorem} {\bf (Ergodic convergence rate of the VM-PADMM)}
\label{th:ergodicPADMM}
Let $\{R_k\}$, $\{S_k\}$ and $\{H_k\}$ be generated by the \emph{VM-PADMM}  and
let $\{(x_k^a,y_k^a)\}$, $\{\tilde \gamma_k^a\}$, $\{(r^a_{k,x},r^a_{k,y},r^a_{k,\gamma})\}$
and $\{(\varepsilon^a_{k,x},\varepsilon^a_{k,y})\}$ be the ergodic sequences defined as in \eqref{def:erg}--\eqref{def:erg3b}.
Let also $C_S$, $C_P$, and $d_0$ be  as in \eqref{CS2},  \eqref{pro:Q122} and
\eqref{def:d0admm}, respectively. Then, there exists a parameter $\sigma_\theta\in (0,1)$ such that, for all $k\geq 1$, there hold $\varepsilon_{k,x}^a,\,\varepsilon_{k,y}^a\geq 0$,
\begin{equation}
 \label{eq:main.ec.admm} 
 \left(
\begin{array}{c}
 r^a_{k,\,x}\\[1mm]
 r^a_{k,\,y}\\[1mm]
 r^a_{k,\,\gamma}\\[1mm]
\end{array} 
 \right)
\in 
 \left( 
 \begin{array}{c} 
 \partial f_{\varepsilon^a_{k,x}}(x_k^a)- A^*\tilde{\gamma}_k^a\\
 \partial g_{\varepsilon^a_{k,y}}(y_k^a)- B^*\tilde{\gamma}_k^a
\\Ax_k^a+By_k^a-b
 \end{array} \right)
 \end{equation}
and
\begin{align}
 \label{ineq:Th_ergodic123}
& \max\left\{\|r_{k,\,x}^a\|_{\X, R_k}^*,\, \|r_{k,\,y}^a\|_{\Y, (B^*H_kB^*+S_k)}^*,\,\|r^a_{k,\,\gamma}\|_{\Gamma,\theta^{-1}H^{-1}_k}^*\right\}
\leq \frac{\sqrt{1+\tau_{\theta}}\mathcal{E}\,d_0 }{k}\,,\\
\label{ine:eps1245}
&\varepsilon^a_{k,x}+\varepsilon^a_{k,y}\leq  \frac{(1+\tau_{\theta})\widehat{\mathcal{E}}d_0^2}{k},
\end{align}
where $\mathcal{E}$ and $\widehat{\mathcal{E}}$   are   as in Theorem~\ref{th:vhpe.eg} with $\sigma=\sigma_\theta$ and $\tau_{\theta}$ is as in  Theorem~\ref{th:maintheoADMM}.
\end{theorem}

\noindent
{\bf Remark.} Given  tolerances $\rho,\varepsilon>0$, Theorem \ref{th:ergodicPADMM} guarantees that
there exist scalars $\varepsilon_x,\varepsilon_y\geq 0$, triples $(x,y,\tilde \gamma)$, $(r_x,r_y,r_\gamma)$
and operators $R\in \M_+^{\X}$, $S\in \M_+^{\Y}$ and $H\in \M_{++}^{\Gamma}$ 
(generated by the VM-PADMM) such that
\begin{align}
 \label{eq:comp.re.e}
 \begin{aligned}
 r_x\in &\, \partial_{\varepsilon_x} f(x)-A^*\tilde\gamma, \quad r_y\in \partial_{\varepsilon_y} g(y)-B^*\tilde \gamma, 
\quad r_\gamma=Ax+By-b,\\
 \max&\left\{\|r_x\|^*_{\X,R}\,,\;\|r_y\|^*_{\Y,(B^*HB+S)}\,,\;\|r_\gamma\|^*_{\Gamma,\theta^{-1}H^{-1}}\right\}\leq \rho,\\
\varepsilon_x+&\varepsilon_y\leq \varepsilon,
\end{aligned}
\end{align}
in at most 
\begin{align}
 \label{eq:comp.ree2}
   \mathcal{O}\left((1+C_S)C_p^2\max\left\{\left\lceil\dfrac{{d_0}}{\rho}\,\right\rceil,
\left\lceil\dfrac{d_0^2}{\varepsilon}\,\right\rceil
\right\}\right)
	\end{align}
iterations, where $C_S, C_P$ and $d_0$ are as in Assumption \ref{ass:b1}, \eqref{pro:Q122} and \eqref{def:d0admm}, respectively.  
Note that while the dependence on the tolerance $\rho$ in \eqref{eq:comp.ree2} is better
than the corresponding one in \eqref{eq:comp.rhoe2} by a factor
of $\mathcal{O}(\rho)$,  the inclusions in \eqref{eq:comp.re.e} are potentially weaker than  the corresponding ones in \eqref{eq:comp.rho.e}.
 The triple $(x,y,\tilde \gamma)$ in \eqref{eq:comp.re.e} can be seen as a $(\rho,\varepsilon)$-approximate solution of the KKT system \eqref{eq:fooc}
with residual $(r_x,r_y,r_\gamma)$.

\subsection{Proof of Theorems \ref{th:maintheoADMM} and \ref{th:ergodicPADMM}}
 \label{subsec:proof}

The main goal of this subsection is to prove Theorems \ref{th:maintheoADMM} and \ref{th:ergodicPADMM} 
by viewing the VM-PADMM as an instance of the VM-HPE framework of Section \ref{sec:smhpe} for
solving \eqref{eq:inc.p} with $T:\Z\tos \Z$ defined by
\begin{align}
 \label{FAB}
 T(z):=\left(\begin{array}{c} \partial f(x)- A^*{\gamma}\\ \partial g(y)- B^* {\gamma}\\ Ax+By-b
\end{array} \right),\qquad \forall z:=(x,y,\gamma)\in \Z
\end{align}
where $\Z:=\X\times \Y\times \Gamma
$ is endowed with the usual inner product 
of vectors $z=(x,y,\gamma), z'=(x',y',\gamma')$:
\begin{align}
 \label{eq:inner.p}
 \inner{z}{z'}_{\Z}:=\inner{x}{x'}_{\X}+\inner{y}{y'}_{\Y}+\inner{\gamma}{\gamma'}_{\Gamma}.
\end{align}
The desired results will then
follow essentially from Theorems \ref{th:alpha} and \ref{th:vhpe.eg}, and from 
the identity
\begin{align}
 \label{eq:ext.z}
 T^{-1}(0)=\Omega^*,
\end{align}
where $T^{-1}(0)$ and $\Omega^*$  are the solution sets defined in \eqref{eq:inc.p} and  \eqref{def:ess}, respectively. 
The following linear operators will be needed in our analysis:
\begin{align}
\label{seminorm} 
M_k:=\left(\begin{array}{ccc} {R_k}& 0& 0\\ 0&  B^*H_kB+S_k& 0 \\ 0& 0& \theta^{-1}{H_k}^{-1}
\end{array} \right):\Z\to \Z\qquad \forall\, k\geq 0,
\end{align}
where $\{R_k\}_{k\geq 1}$, $\{S_k\}_{k\geq 1}$ and $\{H_k\}_{k\geq 1}$
are generated by the VM-PADMM and $R_0\in \M_{+}^{\X}$, $S_0\in \M_{+}^{\Y}$, $H_0\in \M_{++}^{\Gamma}$
 are given in Assumption \ref{ass:b1}.
%
%

We begin by  presenting a preliminary technical result.

\begin{proposition} 
\label{pr:aux}
 Let  $\{(x_k,y_k,\gamma_k)\}$ be generated by the \emph{VM-PADMM} 
 and let $\{\tilde \gamma_k\}$ be defined as in \eqref{xtilde}.
Let also $\{M_k\}$ be defined as in \eqref{seminorm}. 
 %
 %
Then,
\begin{equation} 
\label{aux.0}
 M_k \left( \begin{array}{c} x_{k-1}-x_{k}\\ y_{k-1}-y_{k}\\ {\gamma}_{k-1}-{\gamma}_{k}
\end{array} \right) \in \left( \begin{array}{c} \partial f(x_k)-A^*\tilde{\gamma}_k\\\partial g(y_k)-B^*\tilde{\gamma}_k\\Ax_k+By_k-b
\end{array} \right)\qquad \forall\,k\geq 1.
\end{equation}
%
\end{proposition}
\begin{proof}
From the first order optimality conditions for \eqref{def:tsk-admm}
and \eqref{def:tyk-admm}, we obtain, respectively, 
\begin{align*}
\begin{aligned}
&0\in \partial f(x_k)-A^*\left({\gamma}_{k-1}-H_k(Ax_k+By_{k-1}-b)\right)+{R_k}(x_k-{x}_{k-1}),\\
&0 \in \partial g(y_k)-B^*({\gamma}_{k-1}-H_k( Ax_k+By_k-b))+{S_k}(y_k-{y}_{k-1}),
\end{aligned}
\end{align*}
which, combined with \eqref{xtilde}, yields
\begin{align}
 \label{aux.1}
 & {R_k}(x_{k-1}-x_{k})\in \partial f(x_k)-A^*\tilde{\gamma}_k,\quad  (B^*H_kB+S_k)(y_{k-1}-y_{k})\in \partial g(y_k)-B^*\tilde{\gamma}_k.
\end{align}
On the other hand, \eqref{admm:eqxk} (and the assumption $H_k\in \M_{++}^{\Gamma}$) 
gives 
\begin{align}
 \label{eq:aux.2}
 \theta^{-1}H_k^{-1}(\gamma_{k-1} - \gamma_k)=Ax_k+By_k-b.
\end{align}
Using \eqref{seminorm}, \eqref{aux.1} and \eqref{eq:aux.2}
we obtain \eqref{aux.0}. 
\end{proof}

The next lemma will allow us to use the main results of Section \ref{sec:smhpe} for analyzing
the nonasymptotic convergence of the VM-PADMM.

\begin{lemma}
 \label{lm:222}
The sequence $\{M_k\}_{k\geq 0}$ defined in \eqref{seminorm}, the scalar $C_S$ and the sequence 
$\{c_k\}$ given in \emph{Assumption \ref{ass:b1}} satisfy 
condition \eqref{CS} of \emph{Assumption \ref{ass:a1}}.
\end{lemma}
\begin{proof}
Note that the first condition in \eqref{CS2} is identical to
the first one in \eqref{CS}. To finish the proof, note that the
second condition in \eqref{CS2}, which by Assumption \ref{ass:b1} is assumed to hold for 
$\{R_k\}_{k\geq 0}$, $\{S_k\}_{k\geq 0}$ and $\{H_k\}_{k\geq 0}$, 
combined with the (block) diagonal structure of $M_k$ gives 
the second condition in \eqref{CS} for $\{c_k\}_{k\geq 0}$
and $\{M_k\}_{k\geq 0}$.
\end{proof}

The following two technical results will be used to prove that the VM-PADMM is  an instance of the VM-HPE framework.

\begin{lemma}
\label{lem:deltak}
Let  $\{(x_k,y_k,\gamma_k)\}$, $\{S_k\}$ and $\{H_k\}$ be generated by the \emph{VM-PADMM} and let 
$\{\tilde{\gamma}_k\}$ be defined as in \eqref{xtilde}. Let also $d_0$ be defined as
in \eqref{def:d0admm}.
Then, the following hold:\\
\item [\emph{(a)}] for any $k\geq 1$, we have
\[
 \tilde{\gamma}_k-\gamma_k=\frac{1-\theta}{\theta}(\gamma_k-\gamma_{k-1})+H_k B(y_{k}-y_{k-1}),\quad  \tilde{\gamma}_k-\gamma_{k-1}=\frac{1}{\theta}(\gamma_k-\gamma_{k-1})+H_k B(y_{k}-y_{k-1});
\]
\emph{(b)} we have
\[
 \frac{1}{2}\|y_1-y_0\|_{\Y,S_{1}}^2-\frac{1}{\sqrt{\theta}}\langle B(y_{1}-y_{0}),\gamma_1-\gamma_{0} \rangle_\Gamma   \leq  4\max\left\{{1},\frac{\theta}{2-\theta}\right\} d_0^2;
\] 
 \emph{(c)} for any $t>0$ and  $k\geq 2$, we have
 {\small
\[ \frac{2}{\theta} \left\langle{\gamma_k-\gamma_{k-1}-(1-\theta)(\gamma_{k-1}-\gamma_{k-2}),}{B(y_k-y_{k-1})}\right\rangle_{\Gamma}
 \geq \frac{2t-1-c_{k-1}}{t}\|y_k-y_{k-1}\|_{\Y,S_k}^2-t\|y_{k-1}-y_{k-2}\|_{\Y,S_{k-1}}^2.
\]}
\end{lemma}
\begin{proof}
(a) This item follows trivially from \eqref{admm:eqxk} and \eqref{xtilde}.

(b) First note that 
\begin{align*}
 \nonumber
 0 &\leq \dfrac{1}{2}\left\|\frac{1}{\sqrt{\theta}}(\gamma_1-\gamma_{0})+ H_1B(y_{1}-y_{0})\right\|_{\Gamma,H^{-1}_1}^2\\
 &= \dfrac{1}{2}\| \gamma_1-\gamma_{0}\|_{\Gamma,\theta^{-1}H_1^{-1}}^2+\frac{1}{\sqrt{\theta}}\langle B(y_{1}-y_{0}),\gamma_1-\gamma_{0} \rangle_\Gamma
 +\dfrac{1}{2}\| B(y_{1}-y_{0})\|_{\Gamma,H_1}^2,
\end{align*}
which combined with the property \eqref{eq:500} yields, for all $z^*:=(x^*,y^*,\gamma^*)\in \Omega^*$,
\begin{align*}
\frac{1}{2}\|y_1-y_0\|_{\Y,S_{1}}^2&-\frac{1}{\sqrt{\theta}}\langle B(y_{1}-y_{0}),\gamma_1-\gamma_{0} \rangle_\Gamma\\ &\leq 
\frac{1}{2}\left(\|y_1-y_0\|_{\Y,S_{1}}^2+ \|\gamma_1-\gamma_{0}\|_{\Gamma,\theta^{-1}H_1^{-1}}^2
 +\| B(y_{1}-y_{0})\|_{\Gamma,H_1}^2\right)\\
 &\leq
\|y_1-y^*\|_{\Y,S_{1}}^2+
\|y_0-y^*\|_{\Y,S_{1}}^2+
\|\gamma_1-\gamma^*\|_{\Gamma,\theta^{-1}H_1^{-1}}^2\\
&+\| \gamma_0-\gamma^*\|_{\Gamma,\theta^{-1}H_1^{-1}}^2+
\| B(y_{1}-y^*)\|_{\Gamma,H_1}^2 +
\| B(y_{0}-y^*)\|_{\Gamma,H_1}^2.
\end{align*}
Direct use of the above inequality and \eqref{seminorm}
yields
\begin{equation}
 \label{eq_000000012}
 \frac{1}{2}\|y_1-y_0\|_{\Y,S_{1}}^2-\frac{1}{\sqrt{\theta}}\langle B(y_{1}-y_{0}),\gamma_1-\gamma_{0} \rangle_\Gamma\leq  
 \|z_1-z^*\|^2_{\Z,M_1} +\|z_0-z^*\|^2_{\Z,M_1},
\end{equation}
where $z_0:=(x_0,y_0,\gamma_0)$ and $z_1:=(x_1,y_1,\gamma_1)$. On the other hand, from Proposition \ref{pr:aux}  and \eqref{seminorm} with $k=1$, we have $r_1:=M_1(z_0-z_1)\in T(\tilde z_1)$, where $T$
is given in \eqref{FAB}. Using this fact, \eqref{eq:ext.z}
and the monotonicity of $T$, we obtain $\langle \tilde z_1-z^*, r_1\rangle\geq 0$ for all $z^*=(x^*,y^*,z^*)\in \Omega^*$.
Hence, from the latter inequality, Lemma~\ref{lema_desigualdadesB} with $(z,z_+,\tilde z)=(z_0,z_1,\tilde z_1)$ and $M=M_1$,  we have, for all $z^*=(x^*,y^*,z^*)\in \Omega^*$,
\begin{equation}
 \label{eq_0000000123}
\|z^*-z_0\|_{\Z,M_1}^2\geq \|z^*-z_1\|_{\Z,M_1}^2+\|z_0-\tilde z_1\|_{\Z,M_1}^2-\|z_1-\tilde z_1\|_{\Z,M_1}^2.
\end{equation}
Note now that letting $\tilde z_1:=(x_1,y_1,\tilde \gamma_1)$, it follows from 
\eqref{seminorm}, item  (a) and some direct calculations that
\begin{align} \label{eq:deltakx}
\|z_1&-\tilde z_1\|^2_{\Z,M_1}=\|\gamma_1-\tilde \gamma_1\|^2_{\Gamma,\theta^{-1}H_1^{-1}}=\left\|\frac{1-\theta}{\theta}(\gamma_1-\gamma_{0})+ H_1B(y_{1}-y_{0})\right\|_{\Gamma,\theta^{-1}H^{-1}_1}^2\nonumber\\
 &= \frac{(1-\theta)^2}{\theta^2}\| \gamma_1-\gamma_{0}\|_{\Gamma,\theta^{-1}H_1^{-1}}^2+\frac{2(1-\theta)}{\theta^2}\langle B(y_{1}-y_{0}),\gamma_1-\gamma_{0} \rangle_\Gamma
 +\frac{1}{\theta}\| B(y_{1}-y_{0})\|_{\Gamma,H_1}^2.
\end{align}
Moreover, using \eqref{seminorm} with $k=1$ and item  (a),
we find
\begin{align}
 \label{eq:deltaky}
 \nonumber
 \|z_0-\tilde z_1\|^2_{\Z,M_1}&=\|x_0-x_1\|^2_{\X,R_1}+\|y_0-y_1\|^2_{\Y,(B^*H_1B+S_1)}+\|\gamma_0-\tilde \gamma_1\|^2_{\Gamma,\theta^{-1}H_1^{-1}}\\
  &\geq \|B(y_1-y_0)\|^2_{\Gamma, H_1}+\left\|\frac{1}{\theta}(\gamma_1-\gamma_{0})+H_k B(y_{1}-y_{0})\right\|^2_{\Gamma,\theta^{-1}H_1^{-1}}\nonumber\\
  &\geq \frac{1+\theta}{\theta}\|B(y_1-y_0)\|^2_{\Gamma, H_1}+\frac{1}{\theta^2}\left\|\gamma_1-\gamma_{0}\right\|^2_{\Gamma,\theta^{-1}H_1^{-1}}+\frac{2}{\theta^2}\langle B(y_{1}-y_{0}),\gamma_1-\gamma_{0} \rangle_\Gamma.
\end{align}
Combining the previous two estimates, we obtain
\begin{align}
 \|z_0-\tilde z_1\|^2_{\Z,M_1}&-\|z_1-\tilde z_1\|^2_{\Z,M_1}\nonumber\\
 &\geq \frac{2-\theta}{\theta}\|\gamma_1- \gamma_0\|^2_{\Gamma,\theta^{-1}H_1^{-1}}+
 \frac{2}{\theta}\langle B(y_{1}-y_{0}),\gamma_1-\gamma_{0} \rangle_\Gamma+\| B(y_{1}-y_{0})\|_{\Gamma,H_1}^2\nonumber\\
 &=\frac{1-\theta}{\theta}\|\gamma_1- \gamma_0\|^2_{\Gamma,\theta^{-1}H_1^{-1}}+\left\|\frac{H_1^{-1/2}(\gamma_1- \gamma_0)}{\theta}+ H_1^{1/2}B(y_{1}-y_{0})\right\|^2_{\Gamma}\nonumber\\
 &\geq\frac{1-\theta}{\theta}\|\gamma_1- \gamma_0\|^2_{\Gamma,\theta^{-1}H_1^{-1}}\nonumber.
 \end{align}
 If $\theta\in (0,1]$, then the last inequality implies that 
\begin{equation}\label{ineq_s23}
 \|z_1-\tilde z_1\|^2_{\Z,M_1}\leq \|z_0-\tilde z_1\|^2_{\Z,M_1}.
\end{equation}
Now, if $\theta\in (1,(\sqrt{5}+1)/2)$, we have
\begin{align*}
 \|z_1-\tilde z_1\|^2_{\Z,M_1}- \|z_0-\tilde z_1\|^2_{\Z,M_1}&\leq\frac{\theta-1}{\theta}\|\gamma_1- \gamma_0\|^2_{\Gamma,\theta^{-1}H_1^{-1}}\\
 &\leq \frac{2(\theta-1)}{\theta}\left(\|\gamma_1- \gamma^*\|^2_{\Gamma,\theta^{-1}H_1^{-1}}+\|\gamma_0-\gamma^*\|^2_{\Gamma,\theta^{-1}H_1^{-1}}\right)\\
&\leq \frac{2(\theta-1)}{\theta} \left[\|z_0-{z}^*\|_{\Z,M_1}^2+\|z_1-{z}^*\|_{\Z,M_1}^2\right]
\end{align*}
where   the second inequality is due to property    \eqref{eq:500}, and the last inequality is due to \eqref{seminorm} and definitions of $z_0,z_1$ and ${z}^*$.
Hence, combining the last estimative with \eqref{eq_0000000123},   we obtain 
$$
\|z_1-{z}^*\|_{\Z,M_1}^2\leq \frac{\theta}{2-\theta}\left(1+\frac{2(\theta-1)}{\theta}\right)\|z_0-{z}^*\|_{\Z,M_1}^2=\frac{3\theta-2}{2-\theta}\|z_0-{z}^* \|_{\Z,M_1}^2.
$$ 
Thus, it follows from   \eqref{eq_0000000123}, \eqref{ineq_s23} and the last inequality that 
\begin{equation}\label{eq:457}
\|z_1-{z}^*\|_{\Z,M_1}^2\leq \max\left\{1,\frac{3\theta-2}{2-\theta}\right\}\|z_0-{z}^*\|_{\Z,M_1}^2.
\end{equation}
Since,  $M_1\preceq (1+c_0)M_0\preceq 2M_0$ (see Assumption \ref{ass:b1} and Lemma \ref{lm:222}), the desired inequality follows  from  \eqref{eq_000000012} and \eqref{eq:457},
and definition of $d_0$ in \eqref{def:d0admm}.

(c) Using the first order optimality condition for \eqref{def:tyk-admm}, \eqref{xtilde}
and item (a), we find, for every $k\geq 1$, 
\begin{align*}
 \partial g(y_k)\ni B^*( \tilde\gamma_k-H_kB(y_k-y_{k-1}))-S_k(y_k-y_{k-1})=\frac{1}{\theta}B^*( \gamma_k-(1-\theta)\gamma_{k-1})-S_k(y_k-y_{k-1}). 
\end{align*}
For any $k\geq 2$, using the above inclusion with $k \leftarrow k$ and $k \leftarrow k-1$, the monotonicity of $\partial g$ 
and the property \eqref{eq:545}, we find
%
\begin{align*}
&\frac{1}{\theta} \left\langle{B^*(\gamma_k-\gamma_{k-1})-(1-\theta)B^*(\gamma_{k-1}-\gamma_{k-2})},{y_k-y_{k-1}}\right\rangle_{\Y}\\ 
&\geq \inner{S_k(y_k-y_{k-1})}{y_k-y_{k-1}}_\Y-\inner{S_{k-1}(y_{k-1}-y_{k-2})}{y_k-y_{k-1}}_\Y\\
 &\geq \|y_k-y_{k-1}\|_{\Y,S_k}^2-\dfrac{1}{2t}\|y_{k-1}-y_{k-2}\|_{\Y,S_{k-1}}^2-\dfrac{t}{2}\|y_k-y_{k-1}\|_{\Y,S_{k-1}}^2,\\
 &\geq \left(1-\frac{1+c_{k-1}}{2t}\right)\|y_k-y_{k-1}\|_{\Y,S_k}^2-\dfrac{t}{2}\|y_{k-1}-y_{k-2}\|_{\Y,S_{k-1}}^2,
\end{align*}
where the last inequality is due to  Proposition~\ref{pr:dns} and Assumption \ref{ass:b1}, and so the proof of the lemma follows.
\end{proof}

\begin{lemma}\label{pro:sigma} 
For every $\theta \in (0,(\sqrt{5}+1)/2)$, there exists a parameter $\sigma_\theta\in (0,1)$ such that,  for all $\sigma\in[\sigma_\theta,1)$, 
the matrix \begin{equation*} 
M_\theta(\sigma)= \left[
\begin{array}{cc} 
 \sigma(1+\theta)-1& (\sigma+\theta-1)(1-\theta)\\[2mm]
(\sigma+\theta-1)(1-\theta)&  \sigma-(1-\theta)^2\\
\end{array} \right]
\end{equation*}
is  symmetric  positive definite, and 
\begin{equation}\label{eq:sigm}
\max\{(1-\theta)^2,1-\theta, {1}/{(1+\theta)}\}< \sigma, \quad \frac{(\sigma+\theta-1)\left({4-2\sqrt{2}}\right)}{\sqrt{2}\theta}< \sigma.
\end{equation}
\end{lemma}
\begin{proof}
Since  $M_\theta(\sigma)$ is symmetric,  the proof is immediate by noting that for $\sigma=1$ and for every $\theta \in (0,(\sqrt{5}+1)/2)$,  $M_\theta(\sigma)$ is definite positive and  \eqref{eq:sigm} trivially holds.
\end{proof}

Next we show that the VM-PADMM can be regarded as an instance of the VM-HPE framework.
\begin{proposition}
\label{maincorADMM} 
Let $\{(x_k,y_k,\gamma_k)\}$  be generated by the \emph{VM-PADMM} and let  
$\{\tilde{\gamma}_k\}$ and $\{M_k\}$  be defined as in~\eqref{xtilde} and \eqref{seminorm}, respectively. Let also $d_0$,  $T$,  $\sigma_\theta$ and $\tau_\theta$  be
as in \eqref{def:d0admm}, \eqref{FAB},  Lemma~\ref{pro:sigma}, and  Theorem~\ref{th:maintheoADMM},  respectively.
Define $z_0:=(x_0,y_0,\gamma_0)$, $\eta_0:=\tau_{\theta}d_0^2$ and, for all $k\geq 1$,
\begin{align}
 \label{i645}
  z_k&:=(x_k,y_k,\gamma_k), \quad  
	\tilde z_k:=(x_k,y_k,\tilde \gamma_k), \quad   r_k:=M_k(z_{k-1}-z_k),\\[3mm]
	\eta_k&:= \frac{{\sigma_\theta-(\theta-1)^2}}{\theta^2}\|\gamma_k-\gamma_{k-1}\|_{\Gamma,\theta^{-1}H_k^{-1}}^2+
\frac{\sqrt{2}(\sigma_\theta+\theta-1)}{\theta}\|y_k-y_{k-1}\|_{\Y,S_k}^2.\label{etak}
\end{align}
Then, for all $k\geq 1$,
\begin{align}
 \label{i646}
 \begin{aligned}
 & r_k\in T(\tilde z_k),\\
 & \|z_k-\tilde z_k\|^2_{\Z,M_k}+\eta_k\leq \sigma_\theta\|z_{k-1}-\tilde z_k\|^2_{\Z,M_k}+\eta_{k-1}.
\end{aligned}
\end{align} 
As a consequence, the \emph{VM-PADMM} falls within the \emph{VM-HPE} framework (with input $z_0$, $\eta_0$ and $\sigma=\sigma_\theta$) for solving
\eqref{eq:inc.p} with $T$ as in \eqref{FAB}.
\end{proposition}
\begin{proof}
First note that the inclusion in \eqref{i646} follows
from \eqref{FAB}, \eqref{aux.0} and the definitions of $z_k$, $\tilde z_k$ and $r_k$ in \eqref{i645}.
Now, using \eqref{eq:inner.p}, \eqref{seminorm}, \eqref{i645}
and some direct calculations, we obtain
\begin{align}
 \label{eq:1001}
\nonumber
 \|z_{k-1}-\tilde z_k\|_{\Z,M_k}^2&= \|x_{k-1}-x_k\|_{\X,R_k}^2+\|B(y_{k-1}-y_k)\|_{\Gamma,H_k}^2+\|y_{k-1}-y_k\|_{\Y,S_k}^2\\
 &+\|\gamma_{k-1}-\tilde \gamma_k\|_{\Gamma,\theta^{-1}H^{-1}_k}^2.
\end{align}
Using the same reasoning and Lemma~\ref{lem:deltak}(a), we also find
\begin{align}
 \label{eq:1002}
 \|z_k-\tilde z_k\|^2_{\Z,M_k}&=\|\gamma_k-\tilde \gamma_k\|^2_{\Gamma,\theta^{-1}H_k^{-1}}=\left\|\frac{1-\theta}{\theta}(\gamma_k-\gamma_{k-1})+H_k B(y_{k}-y_{k-1})\right\|_{\Gamma,\theta^{-1}H_k^{-1}}^2.
\end{align} 
Hence, from Lemma \ref{lem:deltak}(a) 
and some algebraic manipulations, we obtain
\begin{align*}
\nonumber
 \sigma_{\theta}\|\gamma_{k-1}-\tilde \gamma_k\|^2_{\Gamma,\theta^{-1} H_k^{-1}}&-\|\gamma_k-\tilde \gamma_k\|^2_{\Gamma,\theta^{-1} H_k^{-1}}=
\sigma_\theta\left\|\frac{1}{\theta}(\gamma_k-\gamma_{k-1})+H_k B(y_{k}-y_{k-1})\right\|^2_{\Gamma, \theta^{-1}H_k^{-1}}\\
& -\left\|\frac{1-\theta}{\theta}(\gamma_k-\gamma_{k-1})+H_k B(y_{k}-y_{k-1})\right\|_{\Gamma,\theta^{-1}H_k^{-1}}^2\\
\nonumber
 &=\frac{\sigma_\theta-(1-\theta)^2}{\theta^2}\|\gamma_{k}-\gamma_{k-1}\|^2_{\Gamma, \theta^{-1}H_k^{-1}}+\frac{\sigma_\theta-1}{\theta}\|B(y_{k}-y_{k-1})\|_{\Gamma,H_k}^2\nonumber\\
 &+\frac{2(\sigma_\theta+\theta-1)}{\theta^2}\inner{\gamma_{k}-\gamma_{k-1}}{B(y_k-y_{k-1})}_{\Gamma},
\end{align*}
which in turn, combined with \eqref{eq:1001} and \eqref{eq:1002}, yields
\begin{align}
 \label{eq:di23}
 \nonumber
\sigma_\theta\|{z_{k-1}} -\tz_{k}\|_{\Z,M_k}^2&- \|{z_k}- {\tz}_k\|_{\Z,M_k}^2= 
\sigma_\theta \|x_{k}-x_{k-1}\|_{\X,R_k}^2 +\sigma_\theta\| y_{k}-y_{k-1}\|_{\Y,S_k}^2\\
&+\frac{\sigma_\theta-(1-\theta)^2}{\theta^2}\|\gamma_{k}-\gamma_{k-1}\|^2_{\Gamma, \theta^{-1}H_k^{-1}}+\frac{\sigma_\theta(\theta+1)-1}{\theta}\|B(y_{k}-y_{k-1})\|_{\Gamma,H_k}^2\nonumber\\
&+\frac{2(\sigma_\theta+\theta-1)}{\theta^2}\inner{\gamma_{k}-\gamma_{k-1}}{B(y_k-y_{k-1})}_{\Gamma},
\end{align}
We will now consider two cases: $k=1$ and $k>1$. In the first case, it follows from
 \eqref{eq:di23} with $k=1$, Lemma~\ref{lem:deltak}(b),  the first inequality in \eqref{eq:sigm} with $\sigma=\sigma_\theta$, and definitions of $\eta_0$ and $\eta_1$ that
\begin{align*}
\sigma_\theta\|{z_{0}} -\tz_{1}\|_{\Z, M_1}^2- \|{z_1}- {\tz}_1\|_{\Z,M_1}^2+\eta_0-\eta_1
&\geq \left[\sigma_\theta-\frac{\sqrt{2}(\sigma_\theta+\theta-1)}{\theta}  +\frac{\sigma_\theta +\theta-1}{\theta^{3/2}} \right]\|y_1-y_0\|_{\Y,S_{1}}^2,\\
&\geq \left[\sigma_\theta+\frac{(\sigma_\theta+\theta-1)\left({2-3\sqrt{2}}\right)}{3\theta}\right]\|y_1-y_0\|_{\Y,S_{1}}^2,
 \end{align*}
where the last inequality is due to $\sqrt{\theta}\leq 3/2$. Hence, since $({2-3\sqrt{2}})/3\geq(2\sqrt{2}-4)/{\sqrt{2}}$,
inequality~\eqref{i646}  for $k=1$ now follows from  the second inequality in \eqref{eq:sigm} with $\sigma=\sigma_\theta$.
On the other hand, assuming $k>1$, from inequality \eqref{eq:di23}, 
Lemma \ref{lem:deltak}(c) with $t=\sqrt{2}$,  the first inequality in \eqref{eq:sigm} with $\sigma=\sigma_\theta$, and   definition of $\{\eta_k\}$ in \eqref{etak}, we have
\begin{align*}
 &\sigma_\theta\|{z_{k-1}} -\tz_{k}\|_{\Z,M_k}^2- \|{z_k}- {\tz}_k\|_{\Z,M_k}^2 +\eta_{k-1}-\eta_k\geq\frac{\sigma_\theta(\theta+1)-1}{\theta}\|B(y_{k}-y_{k-1})\|_{\Gamma,H_k}^2 \\
 &+\frac{\sigma_\theta-(1-\theta)^2}{\theta^2}\|\gamma_{k-1}-\gamma_{k-2}\|^2_{\Gamma, \theta^{-1}H_{k-1}^{-1}}+\frac{2(\sigma_\theta+\theta-1)(1-\theta)}{\theta^2}\inner{\gamma_{k-1}-\gamma_{k-2}}{B(y_k-y_{k-1})}_{\Gamma}\\
 &+ \left[\frac{(\sigma_\theta+\theta-1)\left({2\sqrt{2}-4+1-c_{k-1}}\right)}{\sqrt{2}\theta}+\sigma_\theta\right]\|y_k-y_{k-1}\|_{\Y,S_k}^2.
 \end{align*}
 Since $c_{k-1}\leq 1$ (see Assumption~\ref{ass:b1}), we obtain from  \eqref{eq:sigm} with $\sigma=\sigma_\theta$   that the term inside  bracket is nonnegative. Hence, inequality \eqref{i646} for $k>1$ now follows  from the first statement of Lemma~\ref{pro:sigma}.

 The last statement of the proposition follows
directly from \eqref{i646} and VM-HPE framework's definition.
\end{proof}

We are now ready to prove Theorems \ref{th:maintheoADMM} and  \ref{th:ergodicPADMM}.

\noindent
{\bf Proof of Theorem \ref{th:maintheoADMM}:} Using Proposition \ref{maincorADMM} and Theorem \ref{th:alpha}, we conclude that, for every 
$k\geq 1$, 
 \eqref{eq:main.pc.admm} holds and  there exists $i\leq k$ such that
  \begin{align}
 \label{eq:main.pc.admm0223}
\|M_i(z_{i-1}-z_i)\|_{\Z,M_i}^*\leq 
	\frac{{d_0}}{\sqrt{k}} \sqrt{ \frac{2(1+\sigma_{\theta})C_P(1+\tau_{\theta})+2(1-\sigma_{\theta})\tau_{\theta}}{(1-\sigma_{\theta})}},
\end{align}
 where $\{M_k\}$ and $\{z_k\}$ are defined in \eqref{seminorm} and \eqref{i645}, respectively. Hence, using Proposition \ref{pr:dn}, we obtain  
   \begin{align}
 \label{eq:main.pc.admm021}
 \nonumber
\|M_i(z_{i-1}-z_i)\|_{\Z,M_i}^*&=\|z_{i-1}-z_i\|_{\Z,M_i}\\
&=\left({\|x_{i-1}-x_i\|_{\X,R_i}^2+\|y_{i-1}-y_i\|_{\mathcal{Y},(B^*H_iB+S_i)}^2+\|\gamma_{i-1}-\gamma_i\|_{\Gamma,\theta^{-1}{H_i}^{-1}}^2}\right)^{1/2}.
\end{align}
 On the other hand, using  Proposition \ref{pr:dn}  and the definition in  \eqref{eq:main.pc.admm}, we find  
 \begin{align*}
 \begin{aligned}
 &\|x_{i-1}-x_i\|_{\X,R_i}=\|R_i(x_{i-1}-x_i)\|^*_{\X,R_i}=\|r_{i,x}\|^*_{\X,R_i},\\
 &\|y_{i-1}-y_i\|_{\mathcal{Y},(B^*H_iB+S_i)}=\|(B^*H_iB+S_i)(y_{i-1}-y_i)\|^*_{\mathcal{Y},(B^*H_iB+S_i)}=\|r_{i,y}\|^*_{\Y,(B^*H_iB+S_i)},\\
 &\|\gamma_{i-1}-\gamma_i\|_{\Gamma,\theta^{-1}H_i^{-1}}=\|\theta^{-1}H_i^{-1}(\gamma_{i-1}-\gamma_i)\|^*_{\Gamma,\theta^{-1}H_i^{-1}}=\|r_{i,\gamma}\|^*_{\Gamma,\theta^{-1}H_i^{-1}},
\end{aligned}
\end{align*}
 which, combined with \eqref{eq:main.pc.admm0223} and \eqref{eq:main.pc.admm021}, proves  \eqref{eq:main.pc.admm02}.
\qed

\vgap
\noindent
{\bf Proof of Theorem \ref{th:ergodicPADMM}:} 
Combining  Proposition \ref{maincorADMM} and Theorem \ref{th:vhpe.eg}, and taking into account that 
$r_k^a=(r_{k,\,x}^a,r_{k,\,y}^a, r_{k,\,\gamma}^a)$, we conclude that, for every $k\geq 1$, 
{\small
\begin{equation} \label{ineq:Th_ergodic120932}
 \max\left\{\|r_{k,\,x}^a\|_{\X, R_k}^*,\, \|r_{k,\,y}^a\|_{\Y, (B^*H_kB^*+S_k)}^*,\,\|r^a_{k,\,\gamma}\|_{\Gamma,\theta^{-1}H^{-1}_k}^*\right\}\leq\|   (r_{k,\,x}^a,r_{k,\,y}^a, r_{k,\,\gamma}^a)   \|_{\Z, M_k}^* \leq \frac{\sqrt{1+\tau_{\theta}}\mathcal{E}\,d_0 }{k},
\end{equation}}
\begin{equation}\label{ine:eps124523}
\varepsilon_k^a=\dfrac{1}{k}\left(\sum_{i=1}^k\,\inner{r_{i,x}}{x_i-x_k^a}_{\X}+
\sum_{i=1}^k\,\inner{r_{i,y}}{y_i-y_k^a}_{\Y}+\sum_{i=1}^k\,\inner{r_{i,\gamma}}{\tilde \gamma_i-\tilde \gamma_k^a}_{\Gamma}\right)\leq\frac{(1+\tau_{\theta})\widehat{\mathcal{E}}d_0^2}{k}.
\end{equation}
On the other hand, \eqref{eq:main.pc.admm}, \eqref{def:erg} and \eqref{def:erg2} yield
\begin{align*}
 Ax_k+By_k=r_{k,\gamma}+b,\quad  Ax^a_k+By^a_k=r^a_{k,\gamma}+b. 
 \end{align*}
Additionally,   \eqref{def:erg}, \eqref{def:erg2} and some algebraic manipulations give
\begin{align*}
 \sum_{i=1}^k\inner{\tilde \gamma_i}{r_{i,\gamma}-r_{k,\gamma}^a}_\Gamma=
 \sum_{i=1}^k\inner{\tilde \gamma_i-\tilde \gamma_k^a}{r_{i,\gamma}-r_{k,\gamma}^a}_\Gamma=
 \sum_{i=1}^k\inner{\tilde \gamma_i-\tilde \gamma_k^a}{r_{i,\gamma}}_\Gamma.
 \end{align*}
Hence, combining  the identity in \eqref{ine:eps124523} with the last two displayed equations, we also find
\begin{align}
\nonumber
\varepsilon_k^a= &\dfrac{1}{k}\sum_{i=1}^k\,\Big(\inner{r_{i,x}}{x_i-x_k^a}_{\X}+
\inner{r_{i,y}}{y_i-y_k^a}_{\Y}\Big)+\dfrac{1}{k}\sum_{i=1}^k\inner{\tilde \gamma_i}{r_{i,\gamma}-r^a_{k,\gamma}}_{\Gamma}\\\nonumber
 &=\dfrac{1}{k}\sum_{i=1}^k\,\Big(\inner{r_{i,x}}{x_i-x_k^a}_{\X}+
\inner{r_{i,y}}{y_i-y_k^a}_{\Y}+\inner{\tilde \gamma_i}{Ax_i-Ax_k^a+By_i-By_k^a}_{\Gamma}\Big)\\\nonumber
&=\frac{1}{{k}}\sum_{i=1}^k \inner{r_{i,x}+A^*\tilde \gamma_i}{x_i-x_k^a}_{\X}+
\frac{1}{{k}}\sum_{i=1}^k \inner{r_{i,y}+B^*\tilde \gamma_i}{y_i-y_k^a}_{\Y} =\varepsilon_{k,x}^a+\varepsilon_{k,y}^a,
\end{align}
where the last equality is due to the definitions of $\varepsilon_{k,x}^a$ and $\varepsilon_{k,y}^a$  in \eqref{def:erg3b}.
Therefore, the inequalities in \eqref{ineq:Th_ergodic123} and \eqref{ine:eps1245} now follows from  \eqref{ineq:Th_ergodic120932}
and \eqref{ine:eps124523}, respectively. 

To finish the proof 
of the theorem, note that direct use of Theorem \ref{th:tf}(b) (for $f$ and $g$), \eqref{eq:main.pc.admm} and \eqref{def:erg}--\eqref{def:erg3b}
give $\varepsilon_{k,x}^a,\,\varepsilon_{k,y}^a\geq 0$ and \eqref{eq:main.ec.admm}.
\qed

\appendix

\section{Proof of Theorems \ref{th:alpha} and \ref{th:vhpe.eg}}
\label{sec:app}

We start by presenting the following two Lemmas.

\begin{lemma}
 \label{lema_desigualdadesB}
 For any $z^*,z,z_+,\tilde z\in \Z$ and $M\in \M^{\Z}_{+}$, we have 
 \begin{align*}
 \|z^*-z\|_{\Z,\,M}^2 - \|z^*-z_+\|_{\Z,\,M}^2 = \|z-\tilde{z}\|_{\Z,M}^2-\|z_+-\tilde{z}\|_{\Z,M}^2
   +2\langle \tilde{z}-z^*, M(z-z_+) \rangle_{\Z}.
\end{align*}
\end{lemma}
\begin{proof}
 Direct calculations yield
\begin{align*}
  \|z^*-z\|_{\Z,M}^2-\|z^*-z_+\|_{\Z,M}^2&=2\langle z_+-z^*, M(z-z_+)\rangle_{\Z}+ \|z_+-z\|_{\Z,\,M}^2\\
  & = 2\langle z_+-\tilde z, M(z-z_+)\rangle_{\Z}+2\langle \tilde z-z^*, M(z-z_+)\rangle_{\Z}\\
&+ \|z_+-z\|_{\Z,M}^2 \\
& = 2\langle \tilde{z}-z^*, M(z-z_+)\rangle_{\Z}+ \|\tilde{z}-z\|_{\Z,M}^2- \|\tilde z-z_{+}\|_{\Z,M}^2\,. 
\end{align*}
\end{proof}

\begin{lemma}
\label{lema_desigualdades}
Let $\{z_k\}$, $\{M_k\}$, $\{\tilde z_k\}$ and $\{\eta_k\}$ be generated by the \emph{VM-HPE framework}. For every $k \geq 1$
and $z^* \in T^{-1}(0):$
\begin{itemize}
%
%
 \item[\emph{(a)}] we have
\[  
  \|z^*-z_{k}\|_{\Z,M_k}^2\leq 
  \|z^*-z_{k-1}\|_{\Z,M_k}^2 +\eta_{k-1}-\eta_{k}-(1 - \sigma) \|z_{k-1}-\tilde{z}_k\|_{\Z,M_k}^2;
\]
\item[\emph{(b)}]  we have
\[ 
  \|z^*-z_{k}\|_{\Z,M_k}^2+\eta_k+(1-\sigma) \displaystyle \sum_{i=1}^k\|z_{i-1}-\tilde{z}_i\|_{\Z,M_i}^2 \leq C_P (\|z^*-z_{0}\|_{\Z,M_0}^2  +   \eta_{0})\,,
\]
where $C_P$ and $M_0$ are as in \eqref{pro:Q12}  and \emph{Assumption~\ref{ass:a1}}, respectively.
\end{itemize}
\end{lemma}
\proof 
(a) From Lemma \ref{lema_desigualdadesB} with 
$(z,z_+,\tilde z)=(z_{k-1},z_k,\tilde z_k)$ and $M=M_k$,  \eqref{breg-subpro} and \eqref{breg-cond1}, we obtain
 \[
  \|z^*-z_{k-1}\|_{\Z,M_k}^2 - \|z^*-z_{k}\|_{\Z,M_k}^2 +\eta_{k-1}\ge (1-\sigma) \|z_{k-1}-\tilde{z}_k\|_{\Z,M_k}^2+\eta_k + 
	 2\langle \tilde{z}_k-z^*, r_k\rangle.
 \]
Hence, (a) follows from the above inequality, the fact that $0 \in T(z^*)$ and $r_k \in T(\tz_k)$ (see \eqref{breg-subpro}), and 
the monotonicity of $T$.

(b)  Using (a), \eqref{eq:dnorm} and Assumption \ref{ass:a1},  we find 
\[ 
 \|z^*-z_{k}\|_{\Z,M_k}^2 \leq 
 (1+c_{k-1})\|z^*-z_{k-1}\|_{\Z,\,M_{k-1}}^2+\eta_{k-1}-\eta_{k}-(1 - \sigma) \|z_{k-1}-\tilde{z}_k\|_{\Z,\,M_k}^2.
\]
Thus, the result follows by applying the above inequality recursively and by using  \eqref{pro:Q12}.
\endproof

We are now ready to prove Theorem \ref{th:alpha}.
\\[2mm]
\noindent
{\bf Proof of Theorem \ref{th:alpha}}: First, note that the desired inclusion holds due to \eqref{breg-subpro}. Now, using \eqref{eq:500} and \eqref{breg-cond1}, we obtain, respectively,
\begin{align*}
 \begin{aligned}
 &\|{z}_{k-1}-z_k\|_{\Z,M_k}^2\leq 2\left( 
 \|{z_{k-1}}- {\tz}_k\|_{\Z,M_k}^2+\|{\tz}_k-z_k\|_{\Z,M_k}^2\right),\\
 &\|{\tz}_k-z_k\|_{\Z,M_k}^2 \leq\sigma\|{z_{k-1}}-\tz_{k}\|_{\Z,M_k}^2+\eta_{k-1}-\eta_k.
\end{aligned}
\end{align*}
Combining the above inequalities, 
we find
\[
 \|{z_{k-1}}- {z}_k\|_{\Z,M_k}^2\leq2\left[(1+\sigma) \|{z_{k-1}}- {\tz}_k\|_{\Z,M_k}^2+\eta_{k-1}-\eta_k\right],
\]
which in turn, combined with Lemma~\ref{lema_desigualdades}(b), yields
\begin{align}
 \label{eq:102}
 \sum_{i=1}^k \|{z_{i-1}}- {z}_i\|_{\Z,M_i}^2 \leq \frac{2(1+\sigma)C_P (\|z^*-z_{0}\|_{\Z,M_0}^2  +\eta_{0})+  
2(1-\sigma)\eta_{0}}{ (1 - \sigma)},
\end{align}
for all $z^*\in T^{-1}(0)$.
Hence, \eqref{v_ieps_i-bound-a} follows from Proposition \ref{pr:dn}, \eqref{breg-subpro}, \eqref{def:d0},  \eqref{eq:102} and the fact that  
$\sum_{i=1}^k\,t_i\geq k\min_{i=1,\dots, k} \{t_i\}$.  
\qed

\vgap
Before proceeding to the proof of the ergodic convergence of the VM-HPE framework, let us first present an auxiliary result.
\begin{proposition} 
\label{a2909}
 Let $\{z_k\}$, $\{M_k\}$ and $\{\eta_k\}$ be generated by the \emph{VM-HPE framework} and consider $\{\tilde z_k^a\}$ and $\{\varepsilon_k^a\}$ as in \eqref{SeqErg}. Then, for every $k\geq 1$,
\begin{equation}
 \label{ad56} 
  \varepsilon_k^a \leq \frac{1}{2k}\left(\eta_{0}+\|\tilde z^a_{k}-z_{0}\|_{\Z,M_{0}}^2+
  \sum_{i=1}^{k}c_{i-1}\|\tilde z^a_{k}-z_{i-1}\|_{\Z,M_{i-1}}^2\right),
\end{equation}
where $M_0$ and $\{c_k\}$ are given in \emph{Assumption~\ref{ass:b1}}.
\end{proposition}
\begin{proof}
Using  Lemma~\ref{lema_desigualdadesB} with $(z^*,z,z_+,\tilde z)=(\tilde z^a_{k},z_{i-1},z_i,\tilde z_i)$ and $M=M_i$, \eqref{breg-subpro} and (\ref{breg-cond1}), we find, for every $i=1,\dots,k$,
\begin{align*}
  \|\tilde z^a_{k}-z_{i-1}\|_{\Z,M_i}^2-\|\tilde z^a_{k}-z_{i}\|_{\Z,M_i}^2 +\eta_{i-1}&\ge (1-\sigma)\|\tilde{z}_i-z_{i-1}\|_{\Z,M_i}^2+  
	\eta_i+2\langle r_i, \tilde{z}_i - \tilde z^a_{k}\rangle\\
                                          &\ge\eta_i+2\langle r_i,\tilde{z}_i-\tilde z^a_{k}\rangle, 
\end{align*}
where the second inequality is due to the fact that  $1-\sigma\geq 0$.
Hence, using Assumption \ref{ass:a1} and simple calculations, we obtain 
\[ 
 \|\tilde z^a_{k}-z_{i}\|_{\Z,M_i}^2 \leq (1+c_{i-1})\|\tilde z^a_{k}-z_{i-1}\|_{\Z,M_{i-1}}^2+\eta_{i-1}-\eta_{i}-
 2\langle r_i, \tilde{z}_i -\tilde z^a_{k}\rangle   \quad \forall i=1,\dots,k.
\]
Summing up the last inequality from $i=1$ to $i=k$   and using the definition of $\varepsilon_k^a$  in \eqref{SeqErg}, we have
\[ 
 0\leq\|\tilde z^a_{k}-z_k\|_{\Z,M_k}^2 \leq \sum_{i=1}^{k}c_{i-1}\|\tilde z^a_{k}-z_{i-1}\|_{\Z,M_{i-1}}^2+\|\tilde z^a_{k}-z_{0}\|_{\Z,M_{0}}^2
+\eta_{0}-2 k\,\varepsilon_k^a,
\]
which clearly gives \eqref{ad56}. 
\end{proof}

\noindent
{\bf Proof of Theorem \ref{th:vhpe.eg}}: Note first that the desired inclusion and the first inequality in
\eqref{th:vhpe.eg03} follow from \eqref{breg-subpro}, \eqref{SeqErg} and Theorem \ref{th:tf}(a).
Take $z^*\in T^{-1}(0)$. Now, let us prove the second inequality in \eqref{th:vhpe.eg03}, which will  follow
by bounding the term in the right-hand side of \eqref{ad56}.
Note that, using the convexity of $\|\cdot\|_{M_{i-1}}^2$, inequality  \eqref{eq:500} and~\eqref{SeqErg},
 we find
 \begin{equation} \label{eq:801}
 \|\tilde z_k^a-z_{i-1}\|_{\Z,M_{i-1}}^2\leq \frac1k\sum_{j=1}^{k} \|\tilde z_j-z_{i-1}\|_{\Z,M_{i-1}}^2
\leq\frac2k \sum_{j=1}^{k}\left( \|\tilde z_j-z_j\|_{\Z,M_{i-1}}^2+ \| z_j-z_{i-1}\|_{\Z,M_{i-1}}^2\right).
\end{equation}
From \eqref{pro:Q12}, we have $M_{i-1}\preceq C_PM_j$ for all $j=1,\dots, k$. Hence, using Proposition~\ref{pr:dns},
inequality \eqref{breg-cond1}, Lemma \ref{lema_desigualdades}(b) and \eqref{def:d0}, we find
\begin{align}
  \label{eq:802}
 \nonumber
  \sum_{j=1}^k\,\|\tilde z_j-z_j\|_{\Z,M_{i-1}}^2&\leq C_P\sum_{j=1}^k\,\|\tilde z_j-z_j\|_{\Z,M_{j}}^2  \\
	\nonumber
	&\leq {C_P}\sum_{j=1}^k\,\left(\sigma\|\tilde z_j-z_{j-1}\|^2_{\Z,M_j}+\eta_{j-1}-\eta_j\right)\\
	&\leq \dfrac{\sigma}{1-\sigma}C_p^2(d_0^2+\eta_0)+C_P\eta_0.
\end{align} 
%
On the other hand, using \eqref{eq:500},
$M_{i-1}\preceq C_P M_j$ for all $j=1,\dots, k$, Proposition~\ref{pr:dns}, Lemma \ref{lema_desigualdades}(b)
and \eqref{def:d0}, we obtain
\begin{align}
 \label{eq:803}
 \nonumber
 \sum_{j=1}^k\,\|z_j-z_{i-1}\|^2_{\Z,M_{i-1}}&\leq 2 \sum_{j=1}^k\,\left(\|z_j-z^*\|_{\Z,M_{i-1}}^2+\|z^*-z_{i-1}\|_{\Z,M_{i-1}}^2\right) \\
 \nonumber
&\leq 2 \sum_{j=1}^k\,\left(C_P\|z_j-z^*\|_{\Z,M_{j}}^2+ \|z^*-z_{i-1}\|_{\Z,M_{i-1}}^2\right) \\
 &\leq 2(1+C_P)C_P(d_0^2+\eta_0)k.
\end{align}
It follows from  inequalities  \eqref{eq:801}--\eqref{eq:803} and the fact that $k\geq1$  that
 \[
 \|\tilde z_k^a-z_{i-1}\|_{\Z,M_{i-1}}^2\leq  \left(\frac{\sigma C_P}{1-\sigma}+2(1+C_P)\right)2C_P(d_0^2+\eta_0)+2C_P\eta_0,
\]
which, combined with Proposition~\ref{a2909} and the first condition in \eqref{CS}, yields

\[
\varepsilon^a_{k} \leq\frac1{2k}\left[  2C_P(1+C_S)\left(\frac{\sigma C_P}{1-\sigma}+2(1+C_P)\right)(d_0^2+\eta_0)+\left(1+2(1+C_S)C_P\right)\eta_0\right].
\]
Therefore, the second inequality in  \eqref{th:vhpe.eg03} now follows from definition of $\widehat{\mathcal{E}}$ and  simple calculus.

To finish the proof of the theorem, it remains to prove \eqref{th:vhpe.eg02}. Assume first that $k\geq 2$. Using \eqref{SeqErg} and simple calculus, we have
\begin{equation}
 \label{eq:krk}
 k\,r^a_{k} =\sum_{i=1}^k\,r_i=M_1(z_0-z^*)-M_k(z_k-z^*)+\sum_{i=1}^{k-1} (M_{i+1}-M_i)(z_i-z^*).
\end{equation}
From \eqref{pro:Q12}, we obtain $M_1\preceq C_P M_k$ and $M_1\preceq C_P M_0$. Hence, it follows
from Propositions \ref{pr:dn} and \ref{pr:dns} that
\begin{align}
 \label{eq:300}
 \nonumber
 \|M_1(z_0-z^*)\|^*_{\Z, M_k}&\leq \sqrt{C_P}\|M_1(z_0-z^*)\|^*_{\Z, M_1}\\ 
\nonumber
    &=\sqrt{C_P}\|z_0-z^*\|_{\Z, M_1}\\
		&\leq C_p\|z_0-z^*\|_{\Z,M_0}.
\end{align}
Direct use of Proposition \ref{pr:dn} yields
\begin{align}
 \label{eq:301}
  \|M_k(z_k-z^*)\|_{\Z, M_k}^*=\|z_k-z^*\|_{\Z, M_k}.
\end{align}
Next step is to estimate the general term in the summation in \eqref{eq:krk}. To do this, first note
that using Assumption \ref{ass:a1}, we find  
\begin{equation}
 \label{eq:Li}
  0\preceq  L_i:=M_{i+1}-M_i + c_iM_{i+1}\preceq c_i (2+c_i)M_i\,,\qquad \forall\, i=1,\dots, k-1,
\end{equation}
and so
\begin{align}
 \label{eq:304}
 \nonumber
  \|(M_{i+1}-M_i)(z_i-z^*)\|^*_{\Z, M_k}&=\|(L_i-c_iM_{i+1})(z_i-z^*)\|^*_{\Z, M_k}\\
	&\leq \|L_i(z_i-z^*)\|^*_{\Z, M_k}+c_i\|M_{i+1}(z_i-z^*)\|^*_{\Z, M_k}.
\end{align}
From \eqref{pro:Q12} and the last inequality in \eqref{eq:Li}, we obtain, respectively,
$M_i\preceq C_p M_k$ and $L_i\preceq c_i(2+c_i)M_i$. Hence, using Propositions \ref{pr:dn} and \ref{pr:dns}, 
we have
\begin{align}
 \label{eq:302}
 \nonumber
 \|L_i(z_i-z^*)\|^*_{\Z,M_k}&\leq \sqrt{C_P}\|L_i(z_i-z^*)\|^*_{\Z, M_i}\\
    \nonumber 
		&\leq \sqrt{C_P}\sqrt{c_i(2+c_i)}\|L_i(z_i-z^*)\|^*_{\Z,L_i}\\
		\nonumber 
		&=\sqrt{C_P}\sqrt{c_i(2+c_i)}\|z_i-z^*\|_{\Z,L_i}\\
		&\leq \sqrt{C_P}\,c_i(2+c_i)\|z_i-z^*\|_{\Z, M_i}.
\end{align}
Again, from \eqref{pro:Q12}, we obtain $M_{i+1}\preceq C_P M_k$ and $M_{i+1}\preceq (1+c_i)M_i$, and consequently
 \begin{align}
 \label{eq:303}
 \nonumber
 \|M_{i+1}(z_i-z^*)\|^*_{\Z, M_k}&\leq \sqrt{C_p}\|M_{i+1}(z_i-z^*)\|^*_{\Z, M_{i+1}}\\
  \nonumber
	&=\sqrt{C_p}\|z_i-z^*\|_{\Z, M_{i+1}}\\
	&\leq \sqrt{C_P(1+c_i)}\|z_i-z^*\|_{\Z, M_i}.
\end{align}
Hence, using \eqref{pro:Q12} and \eqref{eq:304}--\eqref{eq:303},
 we find
\begin{align}
 \label{eq:305}
 \nonumber
 \|(M_{i+1}-M_i)(z_i-z^*)\|^*_{\Z,M_k}&\leq c_i\sqrt{C_p}\left(1+(1+c_i)+\sqrt{1+c_i}\,\right)\|z_i-z^*\|_{\Z, M_i}\\
   &\leq c_i\sqrt{C_P}\left(1+C_P+\sqrt{C_P}\right)\|z_i-z^*\|_{\Z, M_i}.
\end{align}
Finally, using  the definition of $d_0$ in  \eqref{def:d0}, \eqref{eq:krk}--\eqref{eq:301}, \eqref{eq:305}  and
Lemma \ref{lema_desigualdades}(b), 
we conclude that
\begin{align*}
 \nonumber
  k\|r^a_k\|_{\Z, M_k}^*&\leq \|M_1(z_0-z^*)\|^*_{\Z, M_k}+\|M_k(z_k-z^*)\|^*_{\Z,M_k}
	+\sum_{i=1}^{k-1}\,\|(M_{i+1}-M_i)(z_i-z^*)\|_{\Z,M_k}^*\\
	\nonumber
	&\leq \left(C_P+1+C_S\sqrt{C_P}\left(1+C_P+\sqrt{C_P}\right)\right)\max_{i=0,\dots, k}\,\|z_i-z^*\|_{\Z,M_i}\\
	\nonumber
	&\leq \sqrt{C_P}\left(C_P+1+C_S\sqrt{C_P}\left(1+C_P+\sqrt{C_P}\right)\right)\sqrt{d_0^2+\eta_0}\\
	&=\left((1+C_P)\left(\sqrt{C_P}+C_SC_P\right)+C_SC_P^{3/2}\right)\sqrt{d_0^2+\eta_0},
\end{align*}
which gives \eqref{th:vhpe.eg02} for the case $k\geq 2$. 
Note now that
by \eqref{pro:Q12}, we have $M_1\preceq C_PM_0$ and so
using Propositions \ref{pr:dn} and  \ref{pr:dns}, Lemma \ref{lema_desigualdades}(b), \eqref{def:d0} and the second identity in \eqref{SeqErg} with $k=1$,
 we find
\begin{align*}
 \|r_1^a\|^*_{\Z,M_1}=\|r_1\|^*_{\Z,M_1}&=\|M_1(z_0-z_1)\|^*_{\Z,M_1}\\
  &=\|z_0-z_1\|_{\Z,M_1}\\
	&\leq \|z_0-z^*\|_{\Z,M_1}+\|z_1-z^*\|_{\Z,M_1}\\
	&\leq \sqrt{C_P}\|z_0-z^*\|_{\Z,M_0}+\|z_1-z^*\|_{\Z,M_1}\\
	&\leq (1+\sqrt{C_P})\sqrt{C_P}\sqrt{d_0^2+\eta_0}\,,
\end{align*} 
which in turn, combined with the fact that $C_P\geq 1$, gives \eqref{th:vhpe.eg02}
for $k=1$.
\qed

\def\cprime{$'$}

\end{document}